\documentclass[11pt,reqno]{amsart}
\usepackage{amsmath,amsfonts,amsthm,amssymb}
\usepackage{graphicx}
\usepackage{verbatim}
\usepackage{color}
\usepackage{amscd}
\usepackage{hyperref}
\usepackage{enumitem}

\usepackage{amsmath}
\newcommand{\stkout}[1]{\ifmmode\text{\sout{\ensuremath{#1}}}\else\sout{#1}\fi}

\definecolor{Purple}{rgb}{0.8,0,0.3}
\definecolor{Teal}{rgb}{0,0.5,0.5}
\definecolor{Navy}{rgb}{0,0,0.5}
\definecolor{Aqua}{rgb}{0,1,1}
\definecolor{Lime}{rgb}{0,1,0}
\definecolor{Green}{rgb}{0,0.5,0}
\definecolor{Olive}{rgb}{0.5,0.5,0}
\begin{document}
\newcommand{\M}{{\mathcal M}}
\newcommand{\loc}{{\mathrm{loc}}}
\newcommand{\core}{C_c^{\infty}(\Omega)}
\newcommand{\sob}{W^{1,p}(\Omega)}
\newcommand{\sobloc}{W^{1,p}_{\mathrm{loc}}(\Omega)}
\newcommand{\merhav}{{\mathcal D}^{1,p}}
\newcommand{\be}{\begin{equation}}
\newcommand{\ee}{\end{equation}}
%\newcommand{\mysection}[1]{\section{#1}\setcounter{equation}{0}}
%%%%%%%%%%%%
\newcommand{\laplace}{\Delta}%{\triangle}
\newcommand{\pl}{\laplace_p}
\newcommand{\grad}{\nabla}%{\bigtriangledown}
\newcommand{\pd}{\partial}
\newcommand{\bo}{\pd}
\newcommand{\csub}{\Subset}
\newcommand{\sm}{\setminus}
\newcommand{\ssm}{:}
\newcommand{\diver}{\mathrm{div}\,}
%%%%%%%%%%
\newcommand\sol[1]{{{#1}}}
%%%%%%%%%%%%%%%
\newcommand{\bea}{\begin{eqnarray}}
\newcommand{\eea}{\end{eqnarray}}
\newcommand{\bean}{\begin{eqnarray*}}
\newcommand{\eean}{\end{eqnarray*}}
\newcommand{\thkl}{\rule[-.5mm]{.3mm}{3mm}}
%%%%%%%%%%%%%%%%%%%%%%%%%%%
\newcommand{\cw}{\stackrel{\rightharpoonup}{\rightharpoonup}}
\newcommand{\id}{\operatorname{id}}
\newcommand{\supp}{\operatorname{supp}}
\newcommand{\wlim}{\mbox{ w-lim }}
\newcommand{\mymu}{{x_N^{-p_*}}}
\newcommand{\R}{{\mathbb R}}
\newcommand{\N}{{\mathbb N}}
\newcommand{\Z}{{\mathbb Z}}
\newcommand{\Q}{{\mathbb Q}}
\newcommand{\abs}[1]{\lvert#1\rvert}
%%%%%%%%%%%
\newtheorem{theorem}{Theorem}[section]
\newtheorem{corollary}[theorem]{Corollary}
\newtheorem{lemma}[theorem]{Lemma}
\newtheorem{proposition}[theorem]{Proposition}
\theoremstyle{definition}
\newtheorem{notation}[theorem]{Notation}
\newtheorem{definition}[theorem]{Definition}
\newtheorem{remark}[theorem]{Remark}
\newtheorem{assertion}[theorem]{Assertion}
\newtheorem{problem}[theorem]{Problem}
\newtheorem{problems}[theorem]{Problems}
%%%%%%%%%%%%%%%%%%
\newtheorem{conjecture}[theorem]{Conjecture}
\newtheorem{question}[theorem]{Question}
\newtheorem{example}[theorem]{Example}
\newtheorem{Thm}[theorem]{Theorem}
\newtheorem{Lem}[theorem]{Lemma}
\newtheorem{Pro}[theorem]{Proposition}
\newtheorem{Def}[theorem]{Definition}
\newtheorem{Exa}[theorem]{Example}
\newtheorem{Exs}[theorem]{Examples}
\newtheorem{Rems}[theorem]{Remarks}
\newtheorem{Rem}[theorem]{Remark}

\newtheorem{Cor}[theorem]{Corollary}
\newtheorem{Conj}[theorem]{Conjecture}
\newtheorem{Prob}[theorem]{Problem}
\newtheorem{Ques}[theorem]{Question}
\newtheorem*{corollary*}{Corollary}
\newtheorem*{theorem*}{Theorem}
\newcommand{\pf}{\noindent \mbox{{\bf Proof}: }}

%%%%%%%%%%%%%%%%%%
%\newenvironment{proof}{{\bf Proof.}}{\hfill $\bowtie$\vskip4mm}

\renewcommand{\theequation}{\thesection.\arabic{equation}} %
%\catcode`@=11 \@addtoreset{equation}{section} \catcode`@=12
%%%%%%%%%%%%%%%%%%%%%%
\newcommand{\Real}{\mathbb{R}}
\newcommand{\real}{\mathbb{R}}
\newcommand{\Nat}{\mathbb{N}}
\newcommand{\ZZ}{\mathbb{Z}}
\newcommand{\CC}{\mathbb{C}}
\newcommand{\Pess}{\opname{Pess}}
\newcommand{\Proof}{\mbox{\noindent {\bf Proof} \hspace{2mm}}}
\newcommand{\mbinom}[2]{\left (\!\!{\renewcommand{\arraystretch}{0.5}
\mbox{$\begin{array}[c]{c}  #1\\ #2  \end{array}$}}\!\! \right )}
\newcommand{\brang}[1]{\langle #1 \rangle}
\newcommand{\vstrut}[1]{\rule{0mm}{#1mm}}
\newcommand{\rec}[1]{\frac{1}{#1}}
\newcommand{\set}[1]{\{#1\}}
\newcommand{\dist}[2]{$\mbox{\rm dist}\,(#1,#2)$}
\newcommand{\opname}[1]{\mbox{\rm #1}\,}
\newcommand{\mb}[1]{\;\mbox{ #1 }\;}
\newcommand{\undersym}[2]
 {{\renewcommand{\arraystretch}{0.5}  \mbox{$\begin{array}[t]{c}
 #1\\ #2  \end{array}$}}}
\newlength{\wex}  \newlength{\hex}
\newcommand{\understack}[3]{%
 \settowidth{\wex}{\mbox{$#3$}} \settoheight{\hex}{\mbox{$#1$}}
 \hspace{\wex}  \raisebox{-1.2\hex}{\makebox[-\wex][c]{$#2$}}
 \makebox[\wex][c]{$#1$}   }%
%%Macros for changing font size in math.
\newcommand{\smit}[1]{\mbox{\small \it #1}}% only for letters, numbers
\newcommand{\lgit}[1]{\mbox{\large \it #1}}% only for letters, numbers
\newcommand{\scts}[1]{\scriptstyle #1}
\newcommand{\scss}[1]{\scriptscriptstyle #1}
\newcommand{\txts}[1]{\textstyle #1}
\newcommand{\dsps}[1]{\displaystyle #1}
%%%%%%%%%%%%%%%%%%%%%%%%%%%%%%%%%%%%%%%%%%%%%%
\newcommand{\dx}{\,\mathrm{d}x}
\newcommand{\dy}{\,\mathrm{d}y}
\newcommand{\dz}{\,\mathrm{d}z}
\newcommand{\dt}{\,\mathrm{d}t}
\newcommand{\dr}{\,\mathrm{d}r}
\newcommand{\du}{\,\mathrm{d}u}
\newcommand{\dv}{\,\mathrm{d}v}
\newcommand{\dV}{\,\mathrm{d}V}
\newcommand{\ds}{\,\mathrm{d}s}
\newcommand{\dS}{\,\mathrm{d}S}
\newcommand{\dk}{\,\mathrm{d}k}

\newcommand{\dH}{\,\mathrm{d}\mathcal{H}^{n-1}}

\newcommand{\dphi}{\,\mathrm{d}\phi}
\newcommand{\dtau}{\,\mathrm{d}\tau}
\newcommand{\dxi}{\,\mathrm{d}\xi}
\newcommand{\deta}{\,\mathrm{d}\eta}
\newcommand{\dsigma}{\,\mathrm{d}\sigma}
\newcommand{\dtheta}{\,\mathrm{d}\theta}
\newcommand{\dnu}{\,\mathrm{d}\nu}

%%%%%%%%%%%%%%%%%%%%%%%%%%%%%%%Macros for Greek letters.
\def\ga{\alpha}     \def\gb{\beta}       \def\gg{\gamma}
\def\gc{\chi}       \def\gd{\delta}      \def\ge{\varepsilon}
\def\gth{\theta}                         \def\vge{\varepsilon}
\def\gf{\phi}       \def\vgf{\varphi}    \def\gh{\eta}
\def\gi{\iota}      \def\gk{\kappa}      \def\gl{\lambda}
\def\gm{\mu}        \def\gn{\nu}         \def\gp{\pi}
\def\vgp{\varpi}    \def\gr{\rho}        \def\vgr{\varrho}
\def\gs{\sigma}     \def\vgs{\varsigma}  \def\gt{\tau}
\def\gu{\upsilon}   \def\gv{\vartheta}   \def\gw{\omega}
\def\gx{\xi}        \def\gy{\psi}        \def\gz{\zeta}
\def\Gg{\Gamma}     \def\Gd{\Delta}      \def\Gf{\Phi}
\def\Gth{\Theta}
\def\Gl{\Lambda}    \def\Gs{\Sigma}      \def\Gp{\Pi}
\def\Gw{\Omega}     \def\Gx{\Xi}         \def\Gy{\Psi}
%%%%%%%%%%%%%%%%%%%%%%%%%%%%%%%%%%%%%%%%%%%%%%%%%%%%%%%%%%%%%%%%%%

\renewcommand{\div}{\mathrm{div}}
\newcommand{\red}[1]{{\color{red} #1}}
\newcommand{\blue}[1]{{\color{blue} #1}}

\newcommand{\eat}[1]{}
\newcommand{\Hmm}[1]{\leavevmode{\marginpar{\tiny%
$\hbox to 0mm{\hspace*{-0.5mm}$\leftarrow$\hss}%
\vcenter{\vrule depth 0.1mm height 0.1mm width \the\marginparwidth}%
\hbox to
0mm{\hss$\rightarrow$\hspace*{-0.5mm}}$\\\relax\raggedright #1}}}

%\begin{titlepage}
%\begin{document}
	\title[Optimal Hardy weights via Green potential]{Strictly positive optimal Hardy weights on manifolds and graphs via Green potentials }
	
	%\iffalse
	\author {Ujjal Das}
	
	\address {Ujjal Das, School of Mathematical Sciences,  National Institute of Science Education and Research, Bhubaneswar, Odisha, India}
	
	\email {udas@niser.ac.in}
	\author[M.~Keller]{Matthias Keller}
	\address{M.~Keller,  Institut f\"ur Mathematik, Universit\"at Potsdam,
		14476  Potsdam, Germany}
	\email{matthias.keller@uni-potsdam.de}	

	\author{Yehuda Pinchover}
	\address{Yehuda Pinchover,
		Department of Mathematics, Technion - Israel Institute of
		Technology,   Haifa, Israel}
	\email{pincho@technion.ac.il}
	
	%\fi
	%%%%%%%%%%%%%%	
	\begin{abstract}
		Consider a subcritical functional $Q$ associated with a $p$-Schr\"od\-inger operator defined on a domain $\Omega$ in a noncompact Riemannian manifold. Employing the supersolution construction, we obtain a family of strictly positive-critical Hardy weights for $Q$. The construction is based on the Green potential {$G_f$, with a strictly positive charge $f$ such that $G_f$ vanishes} at infinity. An interesting new feature of this family,  depending on a parameter $0\leq \varepsilon\leq 1$, is the transition from positive-criticality to null-criticality at $\varepsilon=0$.  {Furthermore, we} prove an analogous result on graphs.
		
		\medskip
		
		\noindent  2020  \! {\em Mathematics  Subject  Classification.}
		Primary  \!  35A23; Secondary 47F05, 49J40. \\[1mm]
		\noindent {\em Keywords:} Criticality theory, Green potential, Hardy inequality, optimal Hardy weight, positive solutions.
	\end{abstract}
	
	\maketitle
	\tableofcontents
	
	\section{Introduction}

Hardy inequalities are a fundamental tool in the analysis of partial differential equations, and have been extensively studied in various contexts. {Agmon's question \cite{Ag82} which asked of finding} Hardy weights for second-order elliptic operators, which are ``as large as possible" was answered in the linear case in \cite{DFP}, and subsequently extended to the quasilinear case in \cite{DP,V}. A {shortcoming} of the {weights constructed there,} is that they are not necessarily strictly positive. Such positivity plays an important role in applications, for example in deriving {spectral results} and Rellich-type inequalities, and, hence, in controlling the growth of solutions. 

The aim of this paper is to construct a one-parameter family of critical Hardy weights which are strictly positive. We employ the supersolution construction with respect to Green potentials with a strictly positive charge. An interesting new feature of this family, which depends on a parameter $0\leq \vge\leq 1$, is that it has a  transition from positive-criticality to null-criticality at $\vge=0$. In particular, the corresponding Hardy inequality is optimal, in the sense of \cite{DP,V} at  $\vge=0$.   We next give a more detailed description of the setting and the problem in the continuum setting and later in Section~\ref{sec_disc} we discuss the discrete case complementing results of \cite{F}.

\section{Setting} \label{setting}
Fix $1<p<\infty$ and a domain $\Gw$ in $\R^n$, {although our results hold also on non-compact smooth Riemannian manifold}.  

Let $A\in L^\infty_\loc(\Gw,\R^{n\times n}){\cap C^{\beta}(\Gw,\R^{n\times n})}$, ${\gb>1}$ be a symmetric locally uniform positive definite matrix, and $V\in L^q_{\mathrm{loc}}(\Omega,\R)$, where {$q>n$}. %\Hmm{Conditions to ensure that solutions are $C^{1,\ga}$} 
The following quasilinear operator $Q'_{A,V}$ on $\Omega$ is called a {\em $p$-Schr\"{o}dinger operator} and is  given by  
\begin{equation*}%\label{eqQ}
Q'(u):=Q'_{A,V}(u)=-\div(|\nabla u|_{A}^{p-2}A\nabla u)+V|u|^{p-2}u,
\end{equation*}
where $|\xi|_{A}:=\sqrt{A(x)\xi\cdot\xi}$.
Denote by
$$Q(\varphi):=Q_{A,V}(\varphi)=\int_\Omega \big(|\nabla \varphi|_A^p+V|\varphi|^p\big)\dx$$ the associated energy functional defined on $C^\infty_c(\Gw)$. We say that the functional $Q$ is {\em nonnegative} 
(in short $Q\geq 0$) in $\Gw$  if $Q(\varphi)\geq 0$ for all  $\varphi\in C^\infty_c    (\Gw)$. The nonnegative functional $Q$ is said to be {\em subcritical} (respectively {\em critical}) in $\Gw$ if there exists (respectively, does not exist) $0\lneqq W\in L^q_{\mathrm{loc}}(\Omega)$, {$q>n/p$,} such that $Q_{A,V-W}\geq 0$  in $\Gw$. Such $W$ is said to be a {\em Hardy weight} for the subcritical functional $Q$ in $\Gw$.

Let $W\gneqq 0$ be a Hardy weight for $Q\geq 0$ in $\Gw$.  Set
\begin{align*}
\lambda_0(Q,W,\Omega)&:=\sup\{\gl\in\Real\mid Q_{A,V-\gl W}\geq 0 \; \mbox{ in }  \Gw \},\\[2mm]
\lambda_\infty(Q,W,\Omega)&:=\sup\{\gl\in\Real\mid \exists K\Subset\Gw\mbox{ s.t. } Q_{A,V-\gl W}\geq 0 \; \mbox{ in }  \Gw\setminus K \},
\end{align*}
which are called respectively, the {\em best Hardy constant} and {\em best Hardy constant at infinity} for  the Hardy-type inequality 
\begin{equation*}%
Q(\varphi) \geq \gl \int_{\Omega} W|\varphi|^p\dx.
\end{equation*}
Next, we introduce the notion of {\em optimal Hardy weights} for $Q$ in $\Gw$. A variant of this definition was first introduced in \cite{DFP} in the linear setting and subsequently extended in \cite{DP,V} to the quasilinear case; in those works, optimal Hardy weights were constructed.
\begin{definition}\label{def_opt}
Suppose that $Q=Q_{A,V}\geq 0$ in $\Gw$. Assume that  $W\in L^q_{\mathrm{loc}}(\Omega)$, {$q>n$},   is a Hardy weight. %Set $\lambda_0:=\lambda_0(Q,W,\Omega)$.

We say that $ W$ is an {\em optimal Hardy weight} for the energy functional $Q$ in $\Gw$ if the following two conditions hold true.
\begin{enumerate}

 \item The functional $Q_{A,V-  W}$ is {\em critical in} $\Gw$, i.e., for any $W_1\gneqq  W$, the Hardy-type inequality
 $$Q(\varphi) \geq \int_{\Omega} W_1 |\varphi|^p\dx $$
does not hold for all $\varphi\in C_c^\infty(\Omega)$. The criticality of $Q_{A,V-  W}$ implies that the equation $Q'_{A,V- W}(u)=0$ in $\Omega$ admits, up to a multiplicative positive constant, a unique positive supersolution $v>0$ which is in fact a solution, and it is called the {\em Agmon ground state}, cf. Lemma~\ref{lem_crit} below.

 \item The functional $Q_{A,V-  W}$ is {\em null-critical} in $\Gw$ with respect to $W$, i.e., it is critical and the Agmon ground state $v$ of $Q_{A,V- W}$ in $\Gw$ satisfies
$$\int_{\Gw} W|v|^p\dx=\infty.$$
 In particular, the variational problem
  \begin{equation*}%
\inf_{{ \varphi} \in \mathcal{D}^{1,\,p}_{Q}(\Omega)} \frac{Q(\vgf)}{\int_{{\Omega}}|\vgf|^pW\dx} 
\end{equation*}
  does not admit a minimizer. Here $\mathcal{D}^{1,\,p}_{Q}(\Omega)$ is the generalized Beppo Levi space, which is defined as the completion of $C_c^\infty(\Omega)$ with respect to the norm $\|\varphi\|_Q:=Q_{A,V^+}(\varphi)^{1/p}$ associated with $Q$ \cite[Definition~1.3]{Das_Pinchover1}.
\end{enumerate}
{A critical functional $Q_{A,V-  W}$ with a ground state $v$ is {\em positive-critical} in $\Gw$ with respect to $W$ if} 
  $$\int_{\Gw} W|v|^p\dx<\infty.$$ 
\end{definition}
\begin{remark}\label{ab_not_c}
(a) Let $W$ be an optimal Hardy weight for $Q=Q_{A,V}$ in $\Gw$. It turns out \cite[Corollary~3.4]{KP} that in this case, $\lambda_\infty\left(Q,W,\Omega \right)=\lambda_0\left(Q,W,\Omega \right)=1$. That is, $\gl_0=1$ is also the best constant for Hardy inequality  restricted to functions $\vgf$ that are compactly supported in any fixed neighborhood of infinity in $\Gw$ (in the sense of one-point compactification).  
 Therefore, the above definition of optimal Hardy weight is equivalent to the definitions in \cite{DFP,DP}.

(b) As mentioned before, an optimal Hardy weight given in \cite{DFP,DP} may vanish somewhere inside $\Gw$.
We note that the fundamental question of identifying sufficient conditions that guarantee the strict positivity of the constructed optimal Hardy weight was not answered in \cite{DFP}, nor in the subsequent literature (see \cite[Remark~9.1]{DFP}).

(c) If $W$ is a critical Hardy weight of $Q_{A,V}$ with a ground state $v$, and let $V_1 \in L^q_{\mathrm{loc}}(\Omega)$, {$q>n$}, such that $V_1+W\gneqq 0$, then clearly, $V_1+W$ is a critical Hardy weight for $Q_{A,V+V_1}$ with a ground state $v$. Moreover, if $\int_{\Gw} (V_1+W)|v|^p\dx=\infty$, then $V_1+W$ is an optimal Hardy weight for $Q_{A,V+V_1}$. 

%(d) A critical functional $Q_{A,V-\gl_0 W}$ is said to be {\em positive-critical} in $\Gw$ with respect to $W$ if  the Agmon ground state $v$ of $Q_{A,V-\gl_0W}$ in $\Gw$ satisfies$$\int_{\Gw} W|v|^p\dx<\infty.$$
\end{remark}

For {$0\lneqq f \in L^{q}_\loc(\Gw), q>n$,} a minimal positive solution $G_f $ of the equation  $Q'_{A,V}(u)=f$ in $\Gw$ is called  {\em Green potential of $Q_{A,V}$ with charge $f$} which is unique by Lemma \ref{lem_Green_pot}. Minimality means that any other positive solution $v$ of the equation  $Q'_{A,V}(u)=f$ satisfies $v\geq G_f$. {We note that under our assumptions $G_f\in C^{1,\ga}(\Gw)$ \cite{Lieberman}.}

Using the supersolution construction approach for Green potentials with compactly supported charge, Idan Versano obtained in \cite{V} optimal Hardy weights for $p$-Schr\"{o}dinger operators with nonzero potential $V$ (see also \cite{DP} where the authors study the $p$-Laplacian case).

\begin{theorem}[{cf. \cite[Theorem~1.1]{V}}]\label{thm:main_non_2}
Let $1<p<\infty$, $\Gw\subset \R^n$, $n\geq 2$, be a domain, $c_0:=(p/(p-1))^{p-1}$, and  $0\lneqq\varphi \in  C_c^\infty(\Gw)$. Let
	$Q_{A,c_0V}$ be subcritical in $\Gw$ { and} $G_{\vgf}$ be the Green potential {of $Q_{A,c_0V}$} with charge $\vgf$  satisfying 
	\begin{equation*}
	\lim\limits_{x\to \bar{\infty}}G_{\vgf}(x)=0, \quad
	\int_{\Gw } |V| G_{\vgf}(x)^{p-1}\dx<\infty, \mbox{ and } \int_{\Gw }V G_{\vgf}(x)^{p-1}\dx < 0,
	\end{equation*}
	where $\bar{\infty}$ denotes the ideal point in the one-point compactification of $\Gw$. 
	Then the operator $Q_{A,V}$ admits an optimal Hardy weight in $\Gw$.
\end{theorem}
A natural and challenging problem emerges from the above result: 

\medskip

\begin{quote}{\em Find a strictly positive optimal Hardy weight. }
\end{quote}

\medskip

A positive answer will imply in the linear case, i.e., $p=2$, various applications in the spectral theory of Schr\"{o}dinger operators (see Section \ref{sec_lin}). 

The  aim of the present paper is to give a positive answer to the above problem. The following result is a special case of Theorem~\ref{pro_crit} proven in Section~\ref{sec_constr}.

\begin{theorem}\label{pro_opt}	
Let $1<p<\infty$, $\Gw\subset \R^n$, $n\geq 2$, be a domain, and $c_0:=(p/(p-1))^{p-1}$. Suppose that $Q_{A,c_0V}$ is subcritical in $\Gw$. Consider  a  Green potential $G_f$ of $Q_{A,c_0V}$ with charge $0 \lneq f\in L^1(\Gw)\cap L^{{q}}_\loc(\Gw)$, {$q > n$}. 
	Let $v_0=(G_f )^{(p-1)/p}$, and 
	$$W_{1}:=\!\frac{Q'_{A,0}(v_0)+Vv_0^{p-1}}{v_0^{p-1}} = \left(\frac{p-1}{p}\right)^p
	\frac{|\nabla G_f |_A^p}{G_f ^p}+
	\left(\frac{p-1}{p}\right)^{p-1} \!\!\frac{f}{G_f ^{p-1}}\,.$$
	Assume that 
	\begin{equation*}
	\lim\limits _{x\to \bar{\infty}}G_{f}=0,	 		
\quad		\int_{\Gw}\! V^-|G_{f}|^{p-1} \dx<\infty, \quad\mbox{ and } 
\quad	\int_{\Gw}\!VG_{f}^{p-1} \dx\leq 0.
	\end{equation*}
 Then $W_1$ is a strictly positive optimal Hardy weight for $Q_{A,V}$ with a ground state $v_0$. 
\end{theorem}

The outline of the present paper is as follows. In the next section, we recall some preliminary results from the criticality theory of $p$-Schr\"{o}dinger operators. In Section~\ref{sec_coarea}, we prove a crucial lemma using the Gauss-Green theorem to bound a quotient of gradients of the Green potential on the boundary of level sets.
In Section~\ref{sec_constr}, we give the proof of Theorem~\ref{pro_crit} which implies Theorem~\ref{pro_opt}. In Section~\ref{sec_disc}, we extend our results to the discrete case. Finally, in Section~\ref{sec_lin}, we discuss some applications of Theorem~\ref{pro_opt} and formulate several open questions arising in the linear case, i.e., when ($p=2$).

\medskip

Throughout the paper, $C$ denotes a positive constant that may change from line to line. Let $f$ and $g$ be positive functions, we write $f\asymp g$ if there exists a positive constant $C$ such that $C^{-1}f\leq g \leq Cf$. For a set $A$, we write $1_A$ for the characteristic function of $A$ and $1$ for the constant function equal to 1. The inclusion $O {\Subset} \Omega$ means that $O$ is open and $\overline{O}$ is compact in $\Omega$. We denote the conjugate {exponent} of $1<p<\infty$ by $p'$, so  $1/p+1/p'=1$. The $(n-1)$-dimensional Hausdorff measure will be denoted by $\mathcal{H}^{n-1}$. The limit $x \rightarrow \bar \infty$  is understood in the topology of the one-point compactification $X \cup \{\infty\}$ of a topological space $X$.

\section{Preliminaries}\label{sec_prelim}
First we recall the well-known Agmon-Allegretto-Piepenbrink (AAP) theorem. Let $Q=Q_{A,V}$ and $Q'=Q_{A,V}'$ be defined as in Section \ref{setting}.
\begin{theorem}[AAP theorem,{ \cite[Theorem~4.3]{PP}}]\label{theorem:AP}  The following assertions are equivalent:
	\begin{enumerate}
	\item $Q\geq 0$  on $ C_c^\infty(\Omega)$.

	\item The equation $Q'(u)=0$  in $\Gw$ admits a positive (weak) solution $v\in  W_{\loc}^{1,p}(\Omega)$, i.e.,  for all $ \varphi \in C_c^{\infty}(\Omega)$,
	  \[\int_{\Omega} |\nabla v|_A^{p-2} {A} \nabla v \cdot \nabla \varphi \dx + \int_{\Omega} V |v|^{p-2} v \varphi \dx =   0 \,. \]
		\item The equation $Q'(u)=0$  in $\Gw$ admits a positive (weak) supersolution $v \in  W_{\loc}^{1,p}(\Omega)$, i.e.,   for all $ 0\leq\varphi \in C_c^{\infty}(\Omega)$,
	  \[\int_{\Omega} |\nabla v|_A^{p-2} {A} \nabla v \cdot \nabla \varphi \dx + \int_{\Omega} V |v|^{p-2} v \varphi \dx  \geq 0  \,. \]
\end{enumerate}
\end{theorem}

Next, we recall the notion of a null-sequence.
 
\begin{Def}[Null-sequence, {\cite[Definition 1.4]{PT1}}]
		We say that a sequence $(\varphi_k)_{k\in\mathbb{N}}$ of nonnegative functions belonging to $W^{1,p}(\Omega)\cap C_{c}(\Gw)$ is a {\em null-sequence} for $Q$ in $\Gw$ if there exists an open set $B {\Subset} \Omega$ such that for all $k\geq 1$
		$$\int_B |\varphi_k|^p\!\dx  \asymp 1 \,,$$
     where the equivalence constants do not depend on $k$,
		and
				$$\lim_{k\to\infty} Q(\varphi_k)=\lim_{k\to\infty} \int_\Omega \!\big(|\nabla \varphi_k|_A^p+V|\varphi_k|^p\big)\dx=0.$$
\end{Def}
If there is a null-sequence, then it converges ({up to a subsequence}) towards a unique Agmon ground state as we recall in the lemma below where we connect the notions of criticality,  ground state, Hardy inequalities and null-sequences.

\begin{lemma}[Characterization of criticality, \cite{PP,PT1}]\label{lem_crit}
Suppose that $Q\geq 0$ in $\Gw$. Then the following assertions are equivalent:
\begin{enumerate}[label=(\roman*)]
\item $Q$ is {\em critical} in $\Omega$.%, i.e., the only nonnegative function $W$ such that the inequality $$Q(\varphi)\geq \int_{\Omega} W|\varphi|^p\dx$$ holds for every $\varphi\in C_c^\infty(\Omega)$ is $W=0$.
\item The equation $Q'(u)=0$ in $\Omega$ admits a unique  positive supersolution (up to a multiplicative constant).
\item  $Q$ admits a null-sequence in $\Gw$.
 \end{enumerate}
Moreover, the supersolution $v$ of (ii) is a  strictly positive solution of $Q'(u)=0$ in $\Omega$, and every null-sequence of (iii)  converges ({up to a subsequence}) {pointwise} and in $L^p_\loc(\Gw)$ to a positive multiple of $v$.
\end{lemma}
%%%%%%%%%%%%%%

We give another characterization of subcriticality in terms of the existence of a Green potential with a positive charge.
\begin{lemma}\label{lem_Green_pot}
	Let $Q=Q_{A,V}\geq 0$ in $\Gw$. {Then $Q$ is subcritical in $\Gw$ if and only if there exists $0<f\in L^q_\loc(\Gw)$, $q>n$ for which $Q'$ admits a Green potential with charge $f$. Moreover, any Green potential with charge $f\geq 0$ is unique.}  	
\end{lemma}	

\begin{proof}
	If $Q$ is critical, then the Agmon ground state satisfies the equation $Q'(u)=0$ in $\Gw$ and it is the unique (up to a multiplicative constant) positive supersolution of the above equation (see Lemma~\ref{lem_crit}). Since any Green potential is a positive supersolution {which is not a solution}, the critical operator $Q'$ does not admit a Green potential. 	
	
	On the other hand, if $Q$ is subcritical, then it admits a strictly positive continuous Hardy weight $W$ \cite[Theorem~4.15 (iii)]{PP} with a positive solution $u$ satisfying the equation $Q'_{A,V-W}(u)=0$ in $\Gw$. Let $f:=Wu^{p-1}$.  For $k\in\N$, consider the unique positive solutions $u_k$ of the   Dirichlet problem 
	\begin{equation*}%\label{eq_seq_gr_pot}
Q'_{A,V}(v)=f|_{\Gw_k} \mbox{ in }\Gw_k, \quad v=0 \mbox{ on } \partial \Gw_k,
\end{equation*}
	where $(\Omega_k)$ is  a smooth, relatively compact exhaustion  of $\Omega$, i.e.,  $\Omega_k\Subset \Omega_{k+1}$ for all $k\geq1 $ and $\bigcup_{k=1}^\infty \Omega_k = \Omega$. Such a sequence converges monotonically to a Green potential $G_f$ with a charge $f$ 
if the above limit is finite  (see for example \cite[Theorem~3.10]{PP}).
 In our case, $u_k$ is bounded by $u$  due to the weak comparison principle (\cite[Theorem 5.3]{PP}). Therefore, the increasing sequence $(u_k)$ converges to a Green potential $G_f$ of $Q'$ satisfying $0<G_f\leq u$. The uniqueness statement follows from the minimality of  Green potential.
\end{proof}
%%%%%%%%%%%%%%%%%%%
A nonnegative functional $Q$ might contain an indefinite term (if the potential has indefinite sign), and in general $Q$ is not convex. Although, by Picone identity \cite{AH}, such functional $Q$ can be represented as the integral of a nonnegative Lagrangian $L$, this $L$ still contains an indefinite term. It was proved in \cite{PTT} (see also  \cite[Lemma 3.4]{PR} for its generalization) that $Q$ is equivalent to a {\em simplified energy} containing only nonnegative terms, as we explain now.
\begin{definition}[Simplified energy]\label{defse}
Let $v \in W^{1,p}_{\loc}(\Omega)$ be a positive solution of the equation $Q'(u)=0$ in $\Gw$. The {\em simplified energy} is defined for nonnegative functions $w\in W^{1,p}(\Omega) \cap C_c(\Omega)$ by
\begin{align*}%\label{eq_simp}
\!Q^v_{\mathrm{sim}}(w)\!:=\!\!\left\{\!\!
  		\begin{array}{ll}
  			\displaystyle{\int_\Omega \!\!\Big(v^2|\nabla w|_A^2\big(v|\nabla w|_A+w|\nabla v|_A\big)^{p-2}\Big)\!\dx} & \!\!\!\mbox{ if } 1<p\leq 2,\\[6mm]
			\displaystyle{\int_\Omega\! \Big(v^p|\nabla w|_A^p+v^2|\nabla v|_A^{p-2}w^{p-2}|\nabla w|_A^2\Big)\!\dx} & \!\!\!\mbox{ if } p>2.
		\end{array}
\right.
    \end{align*}
\end{definition}
We note that Picone identity holds also on manifolds (cf. \cite[Section~2]{PTT}). It follows that Lemma 2.2 in \cite{PTT} (\cite[Lemma 3.4]{PR}) is valid also on manifolds, and hence, we obtain the following equivalence between the functional $Q$ and the simplified energy $Q^v_{\mathrm{sim}}$.
\begin{lemma}[{\cite[Lemma 3.4]{PR}}]\label{simpl_energy}
Assume that $Q=Q_{A,V}\geq 0$ in $\Omega$. Let $v\in W^{1,p}_{\loc}(\Omega)$ be a fixed positive solution of the equation $Q'(u)=0$ in $\Omega$. Then, for all $0\leq w\in W^{1,p}(\Omega) \cap C_c(\Omega)$, we have
$$Q(vw) \asymp Q^v_{\mathrm{sim}}\left(w\right).$$
In particular, for $p=2$, the following well-known identity holds true,
$$Q(vw)=\int_\Omega v^2\left|\nabla w\right|_A^2\dx.$$
\end{lemma}
Lemma \ref{simpl_energy} is a generalization of the \textit{ground state transform}, valid in the linear form (see \cite{DFP}), to the nonlinear case. In the nonlinear case, one obtains the equivalence (and not equality, as in the linear case) between $Q$ and a functional containing only positive terms. As an immediate corollary of Lemma~\ref{simpl_energy}, we state the following obvious {\em upper estimate} for the simplified energy, which will be useful later.

\begin{lemma}[Upper bounds for the simplified energy]\label{simple_est} Let $v\in W^{1,p}_{\loc}(\Omega)$ be a positive solution of the equation $Q'(u)=0$ in $\Omega$, and let $0\leq w\in W^{1,p}(\Omega) \cap C_c(\Omega)$.
Denote
$$X(w):=\int_{\Omega } v^p |\nabla w|_A^p\dx, \qquad Y(w):=\int_{\Omega } |w|^p |\nabla v|_A^p\dx.$$
Then there exists $C=C(p,n)>0$ such that for all $w\in W^{1,p}(\Omega) \cap C_c(\Omega)$ we have
\begin{equation*}%\label{est_qsim}
Q^v_{\mathrm{sim}}(w)\leq C\left\{
  		\begin{array}{ll}
  			X(w) & \mbox{ if } \;1<p\leq 2,\\\\
			X(w)+ {X(w)^{2/p}{Y(w)}^{1-2/p}}  & \mbox{ if }\; p>2.
		\end{array}
\right.
\end{equation*}
%%%%%%%
\end{lemma}
\begin{proof} 
	In the case $1<p\leq 2$, the exponent  $p-2$  in  $Q^v_{\mathrm{sim}}(w)$ is {nonpositive}, and therefore we can drop the term $w|\nabla v|_A$ in the first integral in Definition~\ref{defse} to obtain the desired estimate. In the case $p>2$, we use H\"older's inequality to obtain
	 the required estimate from the definition of $Q^v_{\mathrm{sim}}$. 
\end{proof}

We conclude this section with the following useful application of the chain rule.
\begin{lemma}[{\cite[Lemma~2.10]{DP}}]\label{weak_lapl}
Let ${0<}u\in W_{\mathrm{loc}}^{1,p}(\Omega) \cap C^{1,\alpha}(\Omega)$ {{for some $\alpha\in(0,1)$}}, and $F\in C^2(\R^+)$ {satisfying $F'\geq 0$, $(F')^{p-2}F'' \in C(\R^+)$}. Then the following formula holds in the weak sense:
\begin{equation*}\label{weak_eq}
Q'_{A,0}(F(u))=-|F'(u)|^{p-2}\left[(p-1)F''(u)|\nabla u|_A^p -  F'(u)Q'_{A,0}(u)\right].
\end{equation*}
\end{lemma}
\begin{proof}
	The proof is a direct consequence of the chain rule and the definition of the weak $p$-Laplacian. For the details, we refer to \cite[Lemma 2.10]{DDP}.
\end{proof}
As a direct consequence of Lemma~\ref{weak_lapl}, we get an explicit formula for the Hardy weight obtained by the supersolution construction using powers of Green's potentials. For $1<p<\infty $ and $0\leq\vge\leq 1$,
let
\[\gamma=\gamma_\vge=\frac{p-1+\vge}{p} \,, \ \mbox{ and } \  c_\vge = \gamma^{1-p}\]

\begin{corollary}\label{cor_1}
	Let  $1<p<\infty $ and $0\leq\vge\leq 1$. Assume that $Q_{A,c_\vge V}$ is subcritical in $\Gw$. For $0\lneq f\in L^1(\Gw)\cap L^{q}_\loc(\Gw)$, {$q>n$}, let $G_{f}$  be a Green potential satisfying the equation $Q'_{A,c_\vge V}[G_{f}]=f$ in $\Gw$. 
	Then,
	\begin{equation*}\label{Wvge}
	W_{1+\vge} :=\frac{Q'_{A,V}[G_{f}^\gamma]}{G_{f}^{\gamma(p-1)}}= \gamma^{p-1}
	\left( (\gamma-\vge)\frac{|\nabla G_{f}|_A^p}{G_{f}^p}
	  +\frac{f}{G_{f}^{p-1}} \! \right)
	\end{equation*}
	is positive and nontrivial, 
	and $G_f^{\gamma}$ is a positive solution of the equation
	$Q'_{A,V-W_{1+\vge}}[u] =0$.
In particular, $W_{1+\vge}$ is a Hardy weight for $Q_{A,V}$. Moreover, if $f>0$, then  $W_{1+\vge}>0$ in $\Gw$.
\end{corollary}

%%%%%%%%%%%%%%%%%%%%%% 

\section{Coarea Estimates}\label{sec_coarea}
We denote the strict superlevel sets of the Green potential $G_f$ with a charge $f>0$ and $t>0$ by
$$\Gw_t:=\{x\in \Gw\mid G_{f}(x)>t \}.$$
Furthermore, for a function $V\colon \Gw\to \R$, we denote by $V^\pm$ the positive and negative parts of $V$, respectively, i.e., $V^\pm:=\max\{\pm V,0\}$.

\begin{lemma}\label{coarea grad G} Let $1<p<\infty$ and $\alpha\in(0,1)$.
	Assume that  $Q_{A,V}$ is subcritical in $\Gw$,  and let
	$G_{f}$ be a Green potential with respect to $Q'_{A,V}$ and a charge 
	$0\lneq f\in L^1(\Gw)\cap L^{q}_\loc(\Gw)$, {$q>n$},  and assume that 
	\begin{equation*}%\label{eq:assum_G_G_V}
	\lim\limits _{x\to \bar{\infty}}G_{f}=0,	\qquad  \quad 		
	\quad \quad		\int_{\Gw}V^-G_{f}^{p-1} \dx<\infty.
	\end{equation*}	
	 There exists $C>0$ such that for {a.e. $t>0$}
	\begin{equation*}%\label{eq:coarea_modified}
	\int_{\{G_{f}=t\}}  \frac{|\nabla G_{f}|_A^p}{|\nabla G_{f}|}\dH \leq C,
	\end{equation*}
	Moreover, if in addition, $\int_{\Gw}VG_{f}^{p-1} \dx \leq  0$, then there exists $0<M_{f}<\sup_{\Gw}G_{f}$ such that 
	for a.e.   $0<t<M_{f}$
	$$\int_{\{G_{f}=t\}} \frac{|\nabla G_{f}|_A^p}{|\nabla G_{f}|}\dH \asymp 1,$$
	where the constants do  not depend on $t$.
\end{lemma}
\begin{proof}
Our proof is a careful inspection of  the proof of 
\cite[Proposition 3.1]{DP}. {Since} $G_{f}\in C^{1,\alpha}(\Gw)$ and $$\lim _{x\to \bar{\infty}}G_{f}=0,$$ implies that for a.e. $t>0$   the set 
	$\Gw_t$ has finite perimeter \cite[Theorem~5.9~(i)]{evansf}. 
	Furthermore, by assumption
	$V^-G_{f}^{p-1}\in L^1(\Gw)$. Finally, Sard's theorem for $C^{1,\alpha}$-functions implies that the conditions in Gauss-Green theorem (\cite[Theorems 5.2 and 7.2]{Chen}) are satisfied in $\Gw_t$, and hence,  
\begin{align*}
	-   \int_{\partial \Gw_t}      |\nabla G_{f}|_A^{p-2}   A\nabla G_{f} \cdot  \vec{n} \dH  & = -   \int_{\Gw_t} 
	\diver   (|\nabla G_{f}|_A^{p-2}   A\nabla G_{f}) \dx 
\\ & := -\diver   (|\nabla G_{f}|_A^{p-2}   A\nabla G_{f}) (\Omega_t) \,,
	\end{align*}
    where $\vec{n}$ is the outward unit normal on $\partial \Gw_t$ and $-\diver   (|\nabla G_{f}|_A^{p-2}   A\nabla G_{f})$ acts as a measure. Since $G_f$ is  a solution of  $Q'_{A,V}(G_f)=f$ in distribution sense, it follows that
	\begin{align*}
	 -\diver   (|\nabla G_{f}|_A^{p-2}   A\nabla G_{f}) (\Omega_t)   & = \int_{\Gw_t}(f   -   VG_{f}^{p-1} ) \dx \\& \leq\int_{\Gw_t}(f   +   V^-G_{f}^{p-1} ) \dx \leq C
	\end{align*}
	%The assumption $f+|V| G_{f}^{p-1}\in L^1(\Gw)$ implies 
	%\begin{equation}\label{eq:M} 		-   \int_{\partial \Gw_t}      |\nabla G_{f}|_A^{p-2}   A\nabla G_{f} \cdot  \vec{n} \dH=\int_{\Gw_t}\big (f-VG_{f}^{p-1} \big )\dx \leq  C,	\end{equation}
for some $C>0$ and  for a.e. $t>0$.
	Moreover, if in addition, $\int_{\Gw}VG_{f}^{p-1} \dx \leq  0$, then {$VG_{f}^{p-1}\in L^1(\Gw)$, and therefore,} exists $0<M_{f}<\sup\limits_{\Gw}G_{f}$ such that 
	for almost every  $0<t<M_{f}$ 
	$$
	C^{-1} \leq \int_{\Gw_t} f \dx\leq\int_{\Gw_t}\big (f-VG_{f}^{p-1} \big )\dx.
	$$
	Consequently,
	$$
		-   \int_{\partial \Gw_t}      |\nabla G_{f}|_A^{p-2}   A\nabla G_{f} \cdot  \vec{n} \dH=\int_{\Gw_t}\big (f-VG_{f}^{p-1}\big ) \dx \asymp 1,
	$$
	and $C$ does not depend on $t$.
	Sard's theorem for $C^{1,\alpha}$ functions implies that for almost all~$0<t<\sup_{\Omega}G_{f}$,
	$$\mathcal{H}^{n-1}\left(\{x\in\Omega~|~G_{f}(x)=t~\mbox{and}~\nabla G_{f}(x)=0\}\right)=0.$$
	 Furthermore,  the definition of $\Gw_t$ implies that $G_{f}\geq t$ in 
	$ \overline{\Gw_t}$, and hence, $\vec{n}=-\frac{\nabla G_{f}}{|\nabla G_{f}|}$ for $\mathcal{H}^{n-1}$-a.e. $x\in \partial \Gw_t$. Therefore,
	$$
	-\int_{\partial \Gw_t}|\nabla G_{f}|_A^{p-2}A\nabla G_{f}\cdot \vec{n}\dH =
	\int_{\partial \Gw_t}|\nabla G_{f}|_A^{p-1}\frac{|\nabla G_{f}|_A}{|\nabla G_{f}|}\dH \asymp 1.  \qedhere
	$$ 
\end{proof}

We apply the estimates given by Lemma~\ref{coarea grad G} to obtain the following coarea estimates associated with $Q'_{A,V}$. 

\begin{proposition}[Coarea estimates]\label{pro_key}
Assume the setting of Lemma~\ref{coarea grad G}.
Let $0\leq \eta \leq 1$, $\gamma=\gamma_\eta=(p-1+\eta)/p$ and set $u:=G_f ^{\gamma}$. Then for every {positive}   $g \in  {\in L^{\infty}_{\loc}}((0,\sup_{\Omega} u))$   the following estimate  {holds}
\begin{equation*}%\label{IPP}
\int_{\Omega} g(u) |\nabla u|_A^p \dx \leq C  \int_{0}^{\sup u} {g(t)}t^{\eta/\gamma-1}\,\mathrm{d}t 
\end{equation*}
{for some $C>0$.}
%and
%&\begin{equation*}%\label{IPP2}
%\int_{\Omega} h(G_f ) |\nabla G_f |_A^p \dx \leq C   \int_{\inf G_f }^{\sup G_f } h(t)\dt.
%\end{equation*}
Moreover, if in addition, $\int_{\Gw}VG_{f}^{p-1} \dx \leq  0$, then, for all {positive} $g \in L^{\infty}_{\loc}((0,\sup_{\Omega} u))$ {there exists  $C>0$ such that}
\begin{equation*}%\label{IPP}
\int_{\Omega} g(u) |\nabla u|_A^p \dx {\geq C}  \int_{0}^{{M_f}} g(t)\,t^{\eta/\gamma-1}\,\mathrm{d}t .
\end{equation*}
%and
%\begin{equation*}%\label{IPP2}\int_{\Omega} h(G_f ) |\nabla G_f |_A^p \dx \asymp   \int_{\inf G_f }^{\sup G_f } h(t)\dt.\end{equation*}
\end{proposition}
\begin{proof} The proof is similar to  the proof of Proposition~3.1 in \cite{DP} or Lemma~9.2 in \cite{DFP}.
	By the chain rule,  the co-area formula, and the lemma above    we have %$$\tilde g(x)=g(v(x))v^{-p}(x)\frac{|\nabla G_f|_A^p(x)} {|\nabla G_f|(x)}$$ and
	\begin{multline*}%\label{IPP}
\int_{\Omega} g(u) |\nabla u|_A^p \dx = C\int_{\Omega} g(u) G_f^{-1+\eta}|\nabla G_f|\frac{|\nabla G_f|_A^p} {|\nabla G_f|}\dx \\
=\int_0^\infty g\big(t^{\gamma}\big)t^{-1+\eta} \left(\int_{\{G_f=t\}} \frac{|\nabla G_f|_A^p} {|\nabla G_f|}\dH\right)\mathrm{d}t\\
\leq C \int_{\inf G_f}^{\sup G_f} g\big(t^{{\gamma}}\big)t^{- 1+\eta } \mathrm{d}t= C \int_{\inf u}^{\sup u} {g(\tau)}\tau^{{\eta/\gamma-1}}\,\mathrm{d}\tau,
%\asymp \int_{\Omega} g(v)v^{-p}  |\nabla G_f |^p \dx\\=\int_{\Omega} g(v) |\nabla v|^p \dx =
%\int_{\inf v}^{\sup v} \frac{g(t)}{t}\,\mathrm{d}t
%\int_{\Omega}  g(G_f) |\nabla G_f|_A^p \dx\\=  \int_{\inf G_f }^{\sup G_f }  \tilde g(t)\int_{\{G_f =t\}} \frac{|\nabla G_f |_A^{p} }{|\nabla G_f |}\dH \asymp \int_{\inf G_f }^{\sup G_f }  \tilde g(t)\dt\\=,
\end{multline*}
while if in addition, $\int_{\Gw}VG_{f}^{p-1} \dx \leq  0$, {then the second estimate follows along the same lines.}
\end{proof}

%%%%%%%%%%%%%%%%%%%%%%%%%%%%%%%%%
\section{Construction of null and positive-critical Hardy weights}\label{sec_constr} 
In the present section, we establish the criticality of certain Hardy weights and distinguish between the null-critical and the positive-critical cases. In particular, the following result,  Theorem~\ref{pro_crit}, implies Theorem~\ref{pro_opt}.

Recall the following notation. Let $1<p<\infty$,  $0\leq \vge \leq 1$,  $\gamma=\gamma_{\vge}=(p-1+\vge)/p$, and $c_\vge:={\gamma^{1-p}}$. 
Suppose that $Q'_{A,c_\vge V}$ is subcritical in $\Gw$ which implies that  $Q'_{A,V}$ is also subcritical in $\Gw$  by  \cite[Corollary~4.17]{PP} as $c_\vge \geq 1$ and $Q_{A,0}\geq 0$. 
Let    $G_f<\infty$ be   the Green potential of $Q'_{b,c_\vge V}$ for $0\lneqq f\in L^1(\Gw) \cap L^{p'}_\loc(\Gw)$, i.e., $G_f$ is the unique minimal positive solution of the equation 
	$$ Q'_{A,0}[G_f] +c_\vge V G_f ^{p-1}=f \quad  \mbox{in } \Gw.$$

	We consider  
	\begin{equation*}
		W_{1+\vge} =\frac{Q'_{A,V}[G_f^\gamma]}{G_f^{\gamma(p-1)}}= \gamma^{p-1} 
		\left(  (\gamma-\vge)\frac{|\nabla G_{f}|_A^p}{G_{f}^p}
		+\frac{f}{G_{f}^{p-1}}\right), 
	\end{equation*}
	 as a candidate for the Hardy weight for the functional $Q_{A,V}$.

\begin{theorem}\label{pro_crit}	
	Suppose that $Q_{A,c_\vge V}$ is subcritical in $\Gw$ and let $G_f $ be a Green potential of $Q'_{A,c_\vge V}$ in $\Gw$ for $0\lneq f\in L^1(\Gw)\cap L^{{q}}_\loc(\Gw)$, {$q > n$}. Then, $W_{1+\vge}$ is a Hardy weight for $Q_{A,V}$ which is strictly positive if $f>0$. Assume  that
	\begin{equation*}%\label{eq:assum_G_G_V3}
	\lim\limits _{x\to \bar{\infty}}G_{f}(x)=0,	\qquad 				
	\int_{\Gw}V^-|G_{f}|^{p-1} \dx<\infty.
	\end{equation*}
	 Then,  $Q_{A,V-W_{1+\vge}}$ is critical in $\Gw$ with ground state $v_\vge:= G_f^{\gamma}$  for all $0\leq \vge \leq 1$ and positive-critical with respect to $W_{1+\vge}$  for all $0< \vge \le 1$. %Furthermore, if $f>0$, then $W_{1+\vge}$ is a strictly positive Hardy weight for $0\leq \vge\leq 1$. \medskip
If  additionally  $\int_{\Gw}VG_{f}^{p-1} \dx \leq  0$, then $Q_{A,V-W_{1}}$ is null-critical
 with respect to $W_{1}$.
 %  In particular, if $f>0$, then $W_1$ is a strictly positive optimal Hardy weight for $Q_{A,V}$ with a ground state $v_0$. 
\end{theorem}
\begin{proof}{Fix $0\leq \vge \leq 1$, and let} $v := v_\vge = G_f ^{\gamma}$. Without loss of generality, we may assume that $0<v\leq 1$ in $\Omega$. Then by Corollary~\ref{cor_1}, $v> 0$ satisfies $Q'_{A,0}[v]+(V-W_{1+\vge})|v|^{p-2}v=0$ in $\Omega$, i.e., $Q'_{A,V-W_{1+\vge}}\geq 0$ in $\Gw$.

 \medskip

{\em Criticality:}  We construct a null-sequence for the functional $Q_{A,V-W_{1+\vge}}$ in a similar fashion as in the proof of \cite[Theorem~1.3]{PS}.
{For $k\geq 2$, let}
\begin{align*}
\varphi_{k}(t):=
\begin{cases}
0 & 0\leq t\leq \frac{1}{k^{2}},\\
2+\frac{\log t}{\log k} & \frac{1}{k^{2}}\leq t\leq \frac{1}{k},\\
1 & \frac{1}{k}\leq t,
\end{cases}
\qquad \mbox{and set } \, w_k:=\varphi_k(v).
\end{align*}
Criticality of $Q_{{A,}V-W_{1+\vge}}$ follows from the fact that the sequence $(vw_k)_{k\in\mathbb{N}}$  is a null-sequence  which we show below.\smallskip
 
Since $\lim_{x\to \bar{\infty}}G_{f}=0$, we have that $w_k$ and, therefore, also $vw_k$ are compactly supported in $\Gw$ for all $k\in \N$. 
%Fix a closed ball $B$ in $\Gw$.
By Lemma \ref{simpl_energy} we have
$$Q_{A,V-W_{1+\vge}}(vw)\asymp Q^v_{\mathrm{sim}}(w)$$
 for  $0\leq w\in C_c^\infty(\Omega)$, where $Q^v_{\mathrm{sim}}$ is the simplified energy for the functional $Q_{A,V-W_{1+\vge}}$, associated to $v$ (see Definition~\ref{defse}).
Thus, we need to prove 
\begin{equation*}%\label{eq_null}
\lim_{k\to\infty}\frac{Q^v_{\mathrm{sim}}(w_k)}{\int_{B}(vw_k)^p\dx}=0,
\end{equation*}
where $B\Subset \Gw$ is a fixed ball in $\Gw$. 
Set as in Lemma~\ref{simple_est}
$$X_k:=X(w_k)=\int_{\Omega}v^p|\nabla w_k|_{A}^p\dx,\quad \mbox{and }\quad Y_k:=Y(w_k)=\int_{\Omega}w_k^p|\nabla v|_{A}^p\dx.$$
By the first  coarea estimate, Proposition~\ref{pro_key} with  $u=v$, $\eta=\vge$ and $g(t)=t^p|\varphi_k'(t)|^p {\leq} (\log k)^{-p} $,   we have since we assumed $0<v\leq 1$
%and Lemma~\ref{coarea grad G}, we obtain
\begin{multline*}
X_k=\int_{\Omega}v^{p}|\varphi_k'(v)|^p |\nabla v|_A^p\dx\leq 
C\int_0^1 (t|\varphi_k'(t)|)^p t^{\vge/\gamma-1}{\mathrm{d}t}=\\%[5mm]
 C\left(\frac{1}{\log k}\right)^p \int_{\frac{1}{k^2}}^{\frac{1}{k}}t^{\vge/\gamma-1}{\mathrm{d}t}\leq \displaystyle{
 \begin{cases}
 \dfrac{C}{\left(\log k\right)^{p} k^{\vge /\gamma}} & 0<\vge\leq 1,\\[5mm]
 \dfrac{C}{\left(\log k\right)^{p-1}} & \vge = 0.
 \end{cases}}
 \end{multline*}
 
%where equality in the last estimate holds if $\vge=0$. 
Also, with $g(t)= \varphi _k(t)^p$, as $  \varphi_k \leq 1_{[1/k^2,1]}$, the first  coarea estimate, Proposition~\ref{pro_key}, gives
\begin{multline*}
Y_k=\int_{\Omega}\varphi _k^p(v) |\nabla v|_A^p\dx  \leq C \int_0^1  \varphi_k(t)^pt^{\vge/\gamma-1}{\mathrm{d}t} 
\\\leq C\int_{\frac{1}{k^2}}^{1}t^{\vge/\gamma-1}{\mathrm{d}t}\leq  
\displaystyle{
	\begin{cases}
	C
	& 0<\vge\leq 1,\\[2mm]
	C\log k & \vge = 0.
	\end{cases}}
\end{multline*} 
Furthermore, we clearly have
$$\int_B (vw_k)^p\dx\asymp 1.$$
Recall that by Lemma~\ref{simple_est}, the simplified energy can be estimated from above by
$$Q^v_{\mathrm{sim}}(w_k)\leq C\left\{
  		\begin{array}{ll}
  			X(w_k) & \mbox{ if } 1<p\leq 2,\\\\
			X(w_k)+\ {X(w_k)^{2/p}} Y(w_k)^{1-2/p} & \mbox{ if } p>2.
		\end{array}
\right.$$
Therefore, $\lim_{k\to\infty}Q^v_{\mathrm{sim}}(w_k) =0$.   Thus, $(vw_k) $ is a null-sequence for the functional $Q_{A,V-W_{1+\vge}}$. Hence,  by Lemma~\ref{lem_crit}, for any $0\leq \vge \leq 1$, the functional $Q_{A,V-W_{1+\vge}}$ is critical in $\Gw$  and $v=G_f^\gamma$ is its Agmon ground state. 
% We consider the case $p=2$. In this case $c_\vge =1$, and thus by hypothesis $Q_{A,V}$ is subcritical, i.e., in particular, $Q_{A,V}' \geq 0$. Hence, by AAP theorem there exists a $G>0$ such that $Q_{A,V}'[G]=0$ in $\Gw$. Now, observe that
% \begin{align*}
%     Q_{A,V-W_{1+\vge}}'[\frac{v_\vge}{G}] & = Q_{A,V}'[\frac{v_\vge}{G}] - W_{1+\vge} [\frac{v_\vge}{G}] = \frac{Q_{A,0}'[\frac{v_\vge}{G}]}{G} - W_{1+\vge} \frac{v_\vge}{G} \\
%     & = \frac{Q_{A,0}'[\frac{v_\vge}{G}] - W_{1+\vge} v_\vge}{G} = \frac{Q_{A,0}'[\frac{v_\vge}{G}] - W_{1+\vge} v_\vge}{G}
% \end{align*}
\medskip

{\em Positive-criticality with respect to $W_{1+\vge}$ for $0<\vge\leq1$:} %Let $v:=(G_f)^{\frac{p-1+\vge}{p}}$ be the Agmon ground state of the functional $Q_{A,V-W_{1+\vge}}$. 
The first coarea estimate, Proposition~\ref{pro_key}, applied with $u=G_f$, $\eta=1$ and $g(t)=t^{-1+\vge}$ implies 
\begin{multline*}
\int_{\Gw}  W_{1+\vge} v^p \dx\leq   C(1-\vge)\int_{\Gw} \frac{|\nabla G_f|^p_A}{G_f^{\,p}} G_f^{
	\gamma p}\dx+C\int_\Gw \frac{f}{G_f^{p-1}}G_f^{\gamma p}\dx\\  
\leq C(1-\vge)\int_{0}^{\sup G_f} \frac{\dt}{t^{1-\vge}} + C\int_\Gw fG_f^{\,\vge}\dx< \infty,
\end{multline*}
where the finiteness of the first integral follows as $G_f$ is bounded due to the assumption $\lim_{x\to \bar{\infty}}G_{f}=0$, and the finiteness of the second integral follows from the assumption $f\in L^1(\Gw)$ and the boundedness of $G_f$.
Hence, $Q_{A,V-W_{1+\vge}}$ is positive-critical.
\medskip

{\em Null-criticality with respect to $W_1$:} Let $v=v_0=G_f^{({p-1})/{p}}$. The  coarea estimate, Proposition~\ref{pro_key}, applied with $u=G_f$, $\eta=1$ under the assumption $\int_{\Gw}VG_{f}^{p-1} \dx \leq  0$ and $g(t)=t^{-1}$ on $(0,M_f)$,  implies that
$$\int_{\Gw}  W_1  v^p\dx\geq   C\int_{\Gw} \frac{|\nabla G_f|^p_A}{G_f^{\,p}} G_f^{\,p-1}\dx\geq  C\int_{0}^{M_f} \frac{\dt}{t}
=\infty.$$
This finishes the proof.
\end{proof}
%%%%%%%%%%%
%%%%%%%%%%%%%%%%%%%%
\begin{Rem} [Best constant at infinity]
Recall that by \cite[Corollary~3.4.]{KP}, if $W_1$ is an optimal Hardy weight of $Q$ in $\Gw$, then 
$$\gl_0(Q_{A,V},W_1,\Gw)=\gl_\infty(Q_{A,V},W_1,\Gw)=1,$$
i.e., $1$ is also the best Hardy constant at infinity.  
\end{Rem}
\begin{Rem} [Optimal Hardy weight for $Q_{A,c_0V}$]
Recall that the constant  $c_0=\left( p/(p-1)\right)^{p-1}>1$.  Theorem~\ref{pro_crit}	provides a strictly positive optimal Hardy weight $W_1$ for $Q_{A,V}$ with ground state $v_0$, provided $Q_{A,c_0V}$ is subcritical and an appropriate Green potential of $Q_{A,c_0V}$ satisfies certain conditions. Nevertheless, we can actually construct optimal Hardy weight for the functional $Q_{A,c_0V}$ itself under certain conditions. For instance, our assumptions for $\vge=0$
	$$\int_{\Gw}V^-|G_{f}|^{p-1} \dx<\infty \quad \mbox{ and } \quad \int_{\Gw}V|G_{f}|^{p-1} \dx\leq 0,$$  imply that  $$\int_{\Gw}|V||G_{f}|^{p-1} \dx<\infty.$$
We note that under the further assumption 
		$$\widetilde W_1:=(c_0-1)V+W_1 {>} 0 \,.$$ 
Theorem~\ref{pro_crit} yields the existence of a strictly positive optimal Hardy weight for $Q_{A,c_0V}$ itself:
Indeed, 
$$\int_{\Gw} \widetilde W_1|v_0|^p\dx =\int_{\Gw} ((c_0-1)V+W_1)|v_0|^p\dx=\infty,$$ 
		since $\int_{\Gw}|V||v_0|^p \dx <\infty$ and  $\int_{\Gw}W_1 |v_0|^p\dx =\infty$. Then, Remark~\ref {ab_not_c}.(c) implies that 
		$\widetilde W_1$ is an optimal Hardy weight for $Q_{A,c_0V}$. %Furthermore, if $\widetilde W_1>0$, then  		$\widetilde W_1$ is a strictly positive optimal Hardy weight for $Q_{A,c_0V}$.
\end{Rem}

If $V=0$, we can expand our considerations to the case $1-p< \vge\leq 0$.
\begin{theorem}\label{thm_v0}	
	Let $1<p<\infty$ and $1-p< \vge \leq 1$.	Suppose that $Q_{A,0}$ is subcritical in $\Gw$, and let   $G_f$  be a Green potential of $Q_{A,0}$ for $0\lneq f\in L^1(\Gw)\cap L^{{q}}_\loc(\Gw)$, {$q > n$}. Then, $	W_{1+\vge} :={Q'_{A,0}[G_f^\gamma]}/{G_f^{\gamma(p-1)}}$ is a Hardy weight which is strictly positive if $f>0$. Assume further that  $$\lim\limits _{x\to \bar{\infty}}G_{f}=0.$$
Then, the functional $Q_{A,-W_{1+\vge}}$ is
\begin{enumerate}
	\item[(a)]   subcritical in $\Gw$ for all $1-p< \vge <0$,
	\item[(b)]    critical in $\Gw$ with ground state $G_f^\gamma$  for all $0\leq \vge \leq 1$,
	\item [(c)]    positive-critical with respect to $W_{1+\vge}$  for all $0< \vge \leq 1$,
	\item[(d)]  null-critical with respect to $W_{1}$ for $\vge=0$. In particular, if $f>0$, then $W_{1}$ is a strictly positive optimal Hardy weight for $Q_{A,0}$.
\end{enumerate} 
\end{theorem}  
\begin{proof}
	Note that $W_{1+\vge}$ is an increasing function of $\vge$ in $(1-p,0]$, in particular, for $\vge\in (1-p,0)$, 
	$W_{1+\vge}\lneqq W_1$, and therefore for such $\vge$, $Q_{A,-W_{1+\vge}}$ is subcritical. The other claims of the theorem follow from Theorem~\ref{pro_crit}.    
\end{proof}

\section{The discrete case}\label{sec_disc}

In this section we  discuss the discrete case. By doing so, we complement the results in \cite{F} by presenting a family of critical Hardy weights constructed from Green potentials of subcritical functionals which are optimal at the end of a parameter interval. Our proofs deviate from the ones in \cite{F} by considering also the case $V\neq 0$, and by proving criticality without a bounded oscillation assumption, and hence avoid local finiteness of the underlying graph. {The bounded oscillation assumption is needed only for
null-criticality and for subcriticality in the negative
parameter range when $V=0$.
}%The bounded oscillation assumption is only needed for null-criticality.

Let $X$ be a countable set with the discrete topology and $C(X)$ be the space of all real-valued functions on $X$. We denote by $C_c(X)$ the space of all finitely supported functions on $X$. For any positive or summable function $f$ on a discrete set $A$, we write $\sum_{A} f$ instead of $\sum_{x\in A} f(x)$. 

We equip $X$ with a strictly positive weight function $m:X\to (0,\infty)$ which extends to a measure on $X$ via $m(A)=\sum_{x\in A} m(x)$. Let $b:X\times X\to [0,\infty)$ be a symmetric function such that $b(x,x)=0$ for all $x\in X$, and  which is locally summable, i.e., $\sum_{y\in X} b(x,y)<\infty$ for all $x\in X$. Such a function $b$ is said to be a graph over the measure space $(X,m)$.  We write $x\sim y$ if $b(x,y)>0$ for $x,y\in X$. We assume that the graph $b$ is connected, i.e., for every $ x,y\in  X $ there are $ x= x_{0}  \sim\ldots \sim x_{n}=y$. Furthermore, we let $V:X\to \R$ be a function. We consider the following operator
\begin{align*}
	Q'(u)(x):=Q'_{b,V}(u)(x)=\frac{1}{m(x)}\sum_{y\in X} b(x,y)(\nabla_{xy}u)^{\langle p-1\rangle}+V(x)u(x)^{\langle p-1\rangle},
\end{align*}
where $\nabla_{xy}u=u(x)-u(y)$ and $a^{\langle q\rangle}=|a|^{q-1}a$ for $a\in\R$ and $q\in({0},\infty)$ and {$0^{\langle q\rangle}:=0$}, which acts 
on \begin{align*}\mathcal{F}(X)&:= \{f\in C(X)\mid \sum_{y\in X} b(x,y)|\nabla_{xy}f|^{p-1}<\infty \mbox{ for all } x\in X\}\\
	&= \{f\in C(X) \mid \sum_{y\in X} b(x,y)|f(y)|^{p-1}<\infty  \mbox{ for all } x\in X\}
\end{align*}
where equality holds by the local summability property.
We associate to $Q'$ the following energy functional, 
$Q_{b,V}$ on $C_c(X)$
\begin{align*}
	Q(\varphi):=Q_{b,V}(\varphi)= \frac{1}{2}\sum_{x,y\in X} b(x,y)|\nabla_{xy}\varphi|^p+\sum_{x\in X}m(x)V(x)|\varphi(x)|^p
\end{align*}
for which we assume $$Q(\varphi)\geq 0,\qquad \varphi\in C_c(X).$$ 
{As in the continuous case, we say that $Q_{b,V}$ is
	\emph{subcritical} if there exists a nontrivial nonnegative
	function $W$ such that $Q_{b,V-W}\geq0$ on $C_c(X)$,
	and \emph{critical} otherwise.}

The critical case is characterized by the existence of a null-sequence, i.e., a positive sequence $(w_k)$ in $C_c(X)$ such that $w_k(x)=1$ for some fixed $x$ and $\lim_{k\to \infty} Q_{b,V}(w_k)=0$. Furthermore, any null-sequence converges {pointwise} to the unique (up to a multiplicative constant) positive supersolution of the equation $Q'_{b,V}u=0$ in $X$, which is indeed a solution,  called the {\em Agmon ground state}, cf.~\cite[Theorem~4.1]{Fischer1}. {Let $W$ be a discrete Hardy weight for $Q'_{b,V}\geq 0$ on $X$, and suppose} that $Q_{b,V-W}$ is critical, then it is said to be {\em positive-critical} with respect to $W$ if its ground state belongs to $\ell^p(X, W m)$, and {\em null-critical} otherwise. 

In the subcritical case, there are plenty of positive superharmonic functions, i.e., supersolutions of the equation $Q'_{b,V}[u]=0$ in $X$. For given $f\geq 0$, we denote the minimal positive solution of the equation $Q'_{b,V}[u]=f$ in $X$ by $G_f$ if it exists, and call $G_f$ the \emph{Green potential} of $Q_{b,V}$ in $X$ with charge $f$ in this case.

\subsection{Simplified energy}
The simplified energy replaces the ground state transformation for general $p$ by providing an estimate of the energy functional $Q$ which contains only non-negative terms. 

\begin{definition}[Simplified {Lagrangian} and energy] \label{Def:SEF}
	Let ${v} \in  C(X)$ be a positive function.  For $w\in C_{c}(X) $, let
	\begin{align*}
		S^v({w})_{xy}:= &\ b(x,y) {v}(x) {v}(y) |\nabla_{xy} w|^2 \\
		&\cdot \left[ |\nabla_{xy} {v}|\frac{|w(x)|+|w(y)|}{2} + ({v}(x) {v}(y))^{1/2} |\nabla_{xy} w|\right]^{p-2},
	\end{align*}
	where we set $0 \cdot \infty =0$ in the case $1 < p < 2$. Define the {\em simplified energy functional}  
	of the functional $Q$ with respect to  ${v}$  as
	$$Q^v_{\mathrm{sim}}(w):= \sum_{x,y\in X} S^v({w})_{xy} . $$ 
\end{definition}

The following estimate is found in {\cite[Theorem~3.1 and Eq. (3.5)]{F}}.
\begin{proposition}[Simplified energy estimate] \label{Prop:simp_energyg} For all strictly positive solutions $v$ of the equation   $ Q'[{u}]=0$  and for all  $ w  \in  C_c(X) $
	$$ Q^v_{\mathrm{sim}}(w)\asymp Q({{ v}}w),   $$
	and there is a constant $C>0$ such that
	\begin{align*}
		Q^v_{\mathrm{sim}}(w)\leq C\left\{
		\begin{array}{ll}
			X_v(w) & \mbox{ if } \;1<p\leq 2,\\\\
			X_v(w)+X_v(w)^{2/p}Y_v(w)^{1-2/p}
			& \mbox{ if }\; p>2.
		\end{array}
		\right.
	\end{align*}
	with
	\begin{align*}
		X_v(w)&:=\sum_{x,y\in X} b(x,y)({v}(x){v}(y))^{p/2} |\nabla_{xy} w|^p,\\
		Y_v(w)&:=\sum_{x,y\in X} b(x,y) |\nabla_{xy} {v}|^p \left(\frac{ |w(x)|+|w(y)| }{2}\right)^p.
	\end{align*} 
\end{proposition}
\subsection{Supersolution construction}

Let $1 < p < \infty$, $\varepsilon \in[	0,1]$, and define  $$ \gamma =\gamma_\varepsilon= \frac{p-1+\varepsilon}{p} \in (0,1] \ \ \mbox{ and }\ \ c_\vge=\gamma^{1-p}.$$ 
Let $f\geq 0$ be such that there exists the minimal positive solution  (Green potential with charge $f$) $0<G_f<\infty $  of $Q'_{b,c_\varepsilon V}G_f = f$, where the potential is rescaled by $c_\varepsilon$.

We consider the weight defined by
\begin{equation*}
	W_{1+\varepsilon} := \frac{Q'_{b,V}[G_f^\gamma]}{G_f^{\gamma(p-1)}}.
\end{equation*}
The functions $W_{1+\varepsilon}$ are discrete Hardy weights by AAP theorem \cite[Theorem 2.3]{Fischer1}.
We show an explicit representation of $W_{1+\varepsilon}$ in terms of the Green potential $G_f$ and the charge $f$. This refines the considerations of \cite[Proposition 4.2]{F}.
\begin{lemma}[Representation of $W_{1+\varepsilon	}$] \label{lem_rep_W1ep}
	For $ W_{1+\varepsilon} $, $\varepsilon\in[0,1]$, defined as above, we have
	\begin{align*}%\label{eq:rep_F}
		&W_{1+\varepsilon}(x)= \gamma^{p-1} \frac{f(x)}{G_f(x)^{p-1}} \\
		%	=& \frac{1}{ m(x)} \sum_{y \sim x} b(x,y) \frac{\left|\nabla_{xy} G_f^\gamma\right|^p}{{M}_{p,\gamma}(G_f(x), G_f(y))}\\
		&+  \frac{1}{m(x) G_f(x)^{\gamma p}} \sum_{y \in X} b(x,y){F}_\varepsilon(G_f({y})/ G_f({x})) (\nabla_{xy} G_f)^{\langle p-1 \rangle} 
		{\int_{G_f(y)}^{G_f(x)}  t^{-1+\varepsilon}{\dt}}   \,,
	\end{align*}
	where for $A = G_f(x)$, $B = G_f(y)$, $r=B/A$ the  function ${F}_\varepsilon$ is given for $r\neq 1$
	\begin{align*}
		{F}_\varepsilon(r) &:=\frac{\left( \frac{1-r^\gamma}{1-r} \right)^{p-1}  - \gamma^{p-1} }{\int_{r}^{1}  s^{-1+\varepsilon} \ds }=\frac{\left( \frac{A^\gamma - B^\gamma}{A-B} \right)^{p-1}  A^{\gamma}  - \gamma^{p-1} A^\varepsilon}{\int_{B}^{A}  t^{-1+\varepsilon} \dt }
	\end{align*}
	{and by continuous extension  $F_\varepsilon (1)=\frac{p-1}{2}\gamma^{p-1}(1-\gamma)$ at $r=1$.}
\end{lemma}
\begin{proof}
	Isolating the  potential $V$ from   $Q'_{b,c_\varepsilon V}G_f = f$ yields
	\begin{equation*}%\label{eq:V_isolated}
		V(x) = \gamma^{p-1} \frac{f(x)}{G_f(x)^{p-1}} -\frac{ \gamma^{p-1} }{m(x)}\sum_{y \in X} b(x,y) \frac{(\nabla_{xy} G_f )^{\langle p-1 \rangle}}{G_f(x)^{p-1}}.
	\end{equation*}
	Since $ \gamma p-(p-1) = \varepsilon $, this leads directly to
	\begin{align*} 
		&W_{1+\varepsilon}(x) -\gamma^{p-1} \frac{ f(x)} {G_f(x)^{p-1}} =  \frac{Q'_{b,0}[G_f^\gamma] }{G_f^{\gamma(p-1)}} (x)+ V  (x) -\gamma^{p-1} \frac{ f(x)} {G_f(x)^{p-1}}\\
		=&   \frac{1}{mG_f^{\,\gamma p}(x)} \sum_{y \in  X} b(x,y) \left[ (\nabla_{xy}G_f^\gamma)^{\langle p-1 \rangle}G_f^\gamma(x)   - \gamma^{p-1} (\nabla_{xy}G_f)^{\langle p-1 \rangle} G_f^\varepsilon(x) \right].
	\end{align*} 
	Factoring out the directed gradient component yields the desired representations.
\end{proof}

\begin{lemma}\label{lem_Fbounded}
	For all $1<p<\infty$, $\varepsilon \in {[0,1)}$ and any compact $I\subset (0,\infty)$,  we have
	$$0<\inf_{r\in I} {F_\varepsilon(r)}  \leq \sup_{r\in (0,\infty)} {F_\varepsilon(r)}  <\infty.$$
	Furthermore, $F_1=0$.
\end{lemma}

\begin{proof} For $\varepsilon=1$, we have $\gamma=1$, so the numerator
	in the definition of $F_1$ vanishes identically. Hence $F_1=0$.
	Assume henceforth that $0\le\varepsilon<1$.
	Analyzing the function 
	\begin{align*}
		{F}_\varepsilon(r) &= \frac{\left( \frac{1-r^\gamma}{1-r} \right)^{p-1}  - \gamma^{p-1} }{\int_{r}^{1}  s^{-1+\varepsilon} \ds }  \,,
	\end{align*}
	we see that $F_\varepsilon$ is continuous on $(0,\infty)\setminus\{1\}$. So, it suffices to  check the limits as $t $ goes to $ 0,1$ and $\infty$. They all stay finite by standard arguments. Furthermore, the function $F_\varepsilon $ can also be seen to be positive  on $(0,\infty)$ by concavity of the function $t\mapsto t^\gamma$: Indeed, the numerator compares the $(p-1)$-th powers of the slope of the secant line from $t$ to $1$ to the slope of the tangent line at $1$ which changes sign at $t=1$ and so does the denominator. However, the limits are zero as $t\to 0$ (for $\varepsilon=0$) and $t\to\infty$, so the strict positivity	 of the infimum only  holds on compact subsets $I$ of $(0,\infty)$. 
\end{proof}

\subsection{Coarea estimates}
In this section we present the discrete version of the coarea estimates. Discrete coarea formulas of this type were first obtained for $p=2$ in \cite{KPP} and for $p\neq 2$ in \cite{F}. We present a  version of the coarea estimates under different assumptions.

\begin{lemma}\label{coareagradG}
	Let $1<p <\infty$. Assume $Q'_{b,c_\varepsilon V}$ is subcritical and let $G_f$ be a Green potential with respect to $Q'_{b,c_\varepsilon V}$ and charge $0\lneq f\in \ell^1(X,m)$. Assume that 
	$$\lim_{x\to \bar{\infty}}G_f(x)=0,\qquad \sum_X m{V^-} G_f^{p-1}<\infty.$$
	Then there exists a constant $C>0$ such that for all $t\in (\inf_X G_f, \sup_X G_f)$, we have
	$$h(t) := \sum_{\substack{x,y\in X \\ G_f(y)\le t< G_f(x)}} b(x, y)(\nabla_{x,y}G_f)^{p-1}\leq C.$$
	Moreover, if in addition $\sum_X m V G_f^{p-1}  \leq 0$, then there exists  $0<M_f< \sup_X G_f	$ such that on $(0,M_f)$ we have
	$$h \asymp 1.$$
\end{lemma}
\begin{proof} 
	{	
		For every $t>0$, the superlevel set $\{G_f>t\}$ is finite,
		since $G_f$ vanishes at infinity. Summing the equation for $G_f$
		over this set and cancelling the contributions of internal edges,
		we obtain
		\[
		h(t)
		= \sum_{\{G_f>t\}} m Q'_{b,0}[G_f]
		= \sum_{\{G_f>t\}}
		m\bigl(f-c_\varepsilon VG_f^{p-1}\bigr).
		\]
		Here the superlevel set $\{G_f>t\}$ corresponds to the crossing
		convention $G_f(y)\le t<G_f(x)$ in the definition of $h$. Consequently,
		\[
		0\le h(t)
		\le \|f\|_{1,m}
		+c_\varepsilon\sum_X mV^-G_f^{p-1}
		=: C<\infty,
		\]
		which proves the upper bound.
		
		For the lower bound, since $G_f>0$ and
		$f-c_\varepsilon VG_f^{p-1}\in\ell^1(X,m)$,
		dominated convergence yields
		\[
		\lim_{t\downarrow0}h(t)
		= \|f\|_{1,m}-c_\varepsilon\sum_X mVG_f^{p-1}
		\geq \|f\|_1>0.
		\]
		Hence there exists  $0<M_f<\sup_X G_f$ such that
		\[
		h(t)\geq \frac{\|f\|_{1,m}}{4}
		\qquad\text{for all }0<t<M_f.
		\]
		This proves the lower bound.}
\end{proof}

\begin{proposition}[Coarea estimate]\label{pro_key_g}
	Under the assumptions of Lemma~\ref{coareagradG},  there exists $C\!>\!0$	such that for  {all measurable}  $g\!:\! (\inf_X G_f, \sup_X G_f) \!\to \![0, \infty)$ 
	\begin{equation*}
		\frac{1}{2}\sum_{x,y\in X} b(x, y) (\nabla_{x,y}G_f)^{\langle p-1\rangle} \int_{G_f(y)}^{G_f(x)} g(t) \, dt\le C \int_{\inf G_f}^{\sup G_f} g(t) \ \dt.
	\end{equation*}
	Furthermore, if in addition, $\sum_X m V G_f^{p-1}  \leq 0$, then there exists $0<M_f< \sup_X G_f$ such that, for all {measurable} $g: (0,M_f) \to [0, \infty)$ {extended to $\R$ by $0$}, we have
	\begin{equation*}
		\frac{1}{2}\sum_{x,y\in X} b(x, y) (\nabla_{x,y}G_f)^{\langle p-1\rangle} \int_{G_f(y)}^{G_f(x)} g(t) \, dt\asymp \int_0^{M_f} g(t) \ \dt.
	\end{equation*}
\end{proposition}
\begin{proof} This follows from the appropriate coarea formula for graphs \cite[Theorem~4.7]{F} and Lemma~\ref{coareagradG}.
\end{proof}

\subsection{Null and positive-critical Hardy weights}

Recall that for $1<p<\infty$,  $0\leq \vge \leq 1$, $\gamma=\gamma_\vge:={({p-1+\vge})/{p}}$, and $c_\vge:={\gamma}^{1-p}$.

\begin{theorem}\label{t:graphs}
	Let $1<p <\infty$. Suppose that $Q_{b,c_\vge V}$ is subcritical in $X$, and let $G_f$ be a Green potential with respect to $Q'_{b,c_\vge V}$ and charge $0\lneqq f\in \ell^1(X,m)$. Then, 
	\begin{equation*} 
		W_{1+\vge} :=\frac{Q'_{b,V}[G_f^\gamma]}{G_f^{\gamma(p-1)}} \geq \frac{f}{c_\vge G_f^{p-1}} \geq 0,
	\end{equation*}
	is a Hardy weight for $Q_{b,V}$ which is strictly positive if $f>0$. Assume that 
	\begin{equation*}% \label{eq:assum_G_G_V3g}
		\lim_{x\to \bar{\infty}}G_f(x)=0,\qquad \sum_X mV^-G_f^{p-1}<\infty.
	\end{equation*}
	Then, the functional $Q_{b,V-W_{1+\vge}}$ is critical in $X$ with ground state $G_f^\gamma$  for all $0\leq \vge \leq 1$. Moreover, $Q_{b,V-W_{1+\vge}}$ is positive-critical with respect to $W_{1+\vge}$  for all $0< \vge \le 1$.
	If in addition   $\sum_X mVG_{f}^{p-1}  \leq  0$, and the bounded oscillation assumption $\sup_{x\sim y} G_f(x)/G_f(y)<\infty$ holds, then $Q_{b,V-W_{1}}$ is null-critical
	with respect to $W_{1}$. 
\end{theorem}

\begin{remark}
	A corresponding  statement for $\varepsilon=0$ under slightly different assumptions is found in \cite{F}. Although our proof of criticality is similar in spirit, in contrast it holds for all $\varepsilon \in [0,1]$ and uses different  estimates. Specifically, we do not need the bounded oscillation assumption for criticality and positive-criticality. Furthermore, the proof of positive and null-criticality does use the explicit representation of $W_1$ which did not appear so far in the discrete setting.
\end{remark}

\begin{proof} 
	The estimate for $W_{1+\vge}$ and the corresponding positivity follows from Lemma~\ref{lem_rep_W1ep}.
	
	For $0\leq \vge \leq 1$, and $\gamma=\gamma_\varepsilon={\frac{p-1+\vge}{p}}$, let $v := v_\vge = (G_f )^\gamma$.  Then $v> 0$ satisfies $Q'_{b,V-W_{1+\vge}}v =0$ in $X$, i.e., $Q'_{b,V-W_{1+\vge}} \geq 0$ in $X$.  Without loss of generality, we may assume that $0<v\leq 1$ in $X$.

	{\em Criticality:}  As in the continuum case, we construct a null-sequence for the functional $Q_{b,V-W_{1+\vge}}$.
	For $k\geq 2$, let
	\begin{align} \label{wk}
		\varphi_{k}(t):=
		\begin{cases}
			0 & 0\leq t\leq \frac{1}{k^{2}},\\
			2+\frac{\log t}{\log k} & \frac{1}{k^{2}}\leq t\leq \frac{1}{k},\\
			1 & \frac{1}{k}\leq t,
		\end{cases}
		\qquad \mbox{and \ set } \, w_k:=\varphi_k(v).
	\end{align}
	Criticality of $Q_{b,V-W_{1+\vge}}$ follows from the fact that the sequence $(vw_k)_{k\in\mathbb{N}}$  is a null-sequence,  which we show below.\smallskip
	
	Since $\lim_{x\to \bar{\infty}}G_{f}=0$, we have that $w_k$ and, therefore, also $vw_k$ are compactly (i.e., finitely) supported in $X$ for all $k\in \N$. 
	%Fix a closed ball $B$ in $\Gw$.
	By {Proposition~\ref{Prop:simp_energyg}} we have, for  $0\leq w\in C_c(X)$,
	$$Q_{b,V-W_{1+\vge}} (vw)\asymp Q^v_{\mathrm{sim}}(w),$$
	where $Q^v_{\mathrm{sim}}$ is the simplified energy for the functional $Q_{b,V-W_{1+\vge}}$, associated to $v$ {(see Definition~\ref{Def:SEF} and Proposition~\ref{Prop:simp_energyg}).}
	Note that $w_k(x)=1$ {for sufficienlty large $k$ }at $x$ where $G_f$ takes its maximum. Thus, we only need to   prove that  
	\begin{equation*}%\label{eq_null}
		\lim_{k\to\infty}{Q^v_{\mathrm{sim}}(w_k)}=0 .
	\end{equation*} 
	
	We will use  the upper bound of the simplified energy, Proposition~\ref{Prop:simp_energyg}, and recall the quantities $X_v(w_k)$ and $Y_v(w_k)$ defined therein.
	% it is sufficient to prove that  $\lim_{k\to \infty}X_v(w_k)=0$, and, in the case $p>2$, also that $\sup_{k\in \N }Y_v(w_k)<\infty$, where   
	%\begin{align*} X_v(w_k)&:=\sum_{x,y\in X} b(x,y)\left({v}(x){v}(y)\right)^{\frac{p}{2}} |\nabla_{xy} w_k|^p,\\ Y_v(w_k)&:=\sum_{x,y\in X} b(x,y) \left|\nabla_{xy}{v} \right|^p \left(\frac{ |w_k(x)|+|w_k(y)| }{2}\right)^p.\end{align*}}
	We rewrite $X_v(w_k)$ as
	\begin{align*}
		X_v(w_k) &=\sum_{x,y\in X} b(x,y)(\nabla_{x,y}G_f)^{\langle p-1\rangle} a_k(x,y)\left( \int_{G_f(y)}^{G_f(x)} t^{-1 +\gamma p  } |\varphi'_k(t^{\gamma})|^p\dt\right),
	\end{align*} 
	where 
	\begin{align*}
		a_k(x,y):= \frac{\left({v}(x){v}(y)\right)^{\frac{p}{2}} |\nabla_{xy} w_k|^p}
		{(\nabla_{x,y}G_f)^{\langle p-1\rangle}\left( \int_{G_f(y)}^{G_f(x)} t^{-1 +\gamma p   }|\varphi'_k(t^{\gamma})|^p \dt\right)}
	\end{align*}
	whenever the denominator does not vanish and $a_k(x, y) = 0$, otherwise.
	In addition, for $p\geq 2$, we rewrite $	Y_v(w_k)$ as
	\begin{align*}
		Y_v(w_k)&=	 \sum_{x,y\in X}b(x,y)(\nabla_{x,y}G_f)^{\langle p-1\rangle} \tilde{a}_k(x,y)\left( \int_{G_f(y)}^{G_f	(x)} t^{-1  }1_{ [1/k^{2},1]}(t^\gamma) \dt\right),
	\end{align*} 
	where 
	\begin{align*}
		\tilde{a}_k(x,y):= \frac{\left|\nabla_{x,y}v \right|^p \left(\frac{ |w_k(x)|+|w_k(y)|}{2}\right)^p}
		{(\nabla_{x,y}G_f)^{\langle p-1\rangle}\left( \int_{G_f(y)}^{G_f(x)} t^{-1  }1_{ [1/k^{2},1]}(t^\gamma) \dt\right)}
	\end{align*}
	whenever the denominator does not vanish and $\tilde{a}_k(x, y) = 0$, otherwise.

	It follows from Lemma~\ref{lem:a_bounded} below that ${a}_k$ and $\tilde{a}_k$ are bounded from above. 
	Using the first  coarea estimate, Proposition~\ref{pro_key_g}, with  $g(t)=t^{\gamma p -1}|\varphi_k'(t^\gamma)|^p \leq t^{-1}(\log k)^{-p} $, we have, since we assumed $0<v\leq 1$,
	%and Lemma~\ref{coarea grad G}, we obtain
	\begin{align*}
		X_v(w_k)\leq 
		C\int_0^1 t^{-1 +\gamma p} |\varphi'_k(t^{\gamma})|^p\dt\leq
		\dfrac{C}{\left(\log k\right)^{p}}  \int_{\frac{1}{k^{2/\gamma}}}^{\frac{1}{k^{1/\gamma}}}t^{-1}\dt \leq  \dfrac{C \gamma^{-1}}{\left(\log k\right)^{p-1}}.
	\end{align*}
	Also, with $g(t)=  t^{-1} 1_{[1/k^2,1]} (t^{\gamma})$, the first  coarea estimate, Proposition~\ref{pro_key_g}, gives
	\begin{align*}
		Y_v(w_k)\leq C \int_{0}^{\sup G_f} t^{-1  }1_{ [1/k^{2},1]}(t^\gamma) \dt   \leq  C \int_{\frac{1}{k^{2/\gamma}}}^1   {t}^{-1}{\mathrm{d}t} 
		%\leq 
		%C\int_{\frac{1}{k^2}}^{1}{t}^{-\frac{(p-1)(1-\vge)}{p-1+\vge}}{\mathrm{d}t}
		\leq  
		C \gamma^{-1}\log k .
	\end{align*} 
	By Proposition~\ref{Prop:simp_energyg}, the simplified energy can be estimated from above by
	$$Q^v_{\mathrm{sim}}(w_k)\leq C\left\{
	\begin{array}{ll}
	X_v(w_k) & \mbox{ if } 1<p\leq 2,\\\\
	X_v(w_k)+\ {X_v(w_k)}^{2/p} Y_v(w_k)^{1-2/p} & \mbox{ if } p>2.
	\end{array}
	\right.$$
	Therefore, $\lim_{k\to\infty}Q^v_{\mathrm{sim}}(w_k) =0$.   Thus, $(vw_k) $ is a null-sequence for the functional $Q_{b,V-W_{1+\vge}}$. Hence, $Q_{b,V-W_{1+\vge}}$ is critical in $X$ for any $0\leq \vge \leq 1$ by \cite[Theorem~4.1]{Fischer1}.

	{\em Positive-criticality for $W_{1+\vge}$ for $0<\vge\leq1$:} Recall that $v:=G_f^{\gamma}$ is the Agmon ground state of the operator $Q'_{b,V-W_{1+\vge}}$. We apply the formula for $W_{1+\varepsilon}$ from Lemma~\ref{lem_rep_W1ep} and use the fact that $F_\varepsilon $ is bounded by Lemma~\ref{lem_Fbounded}. This together with the first coarea estimate, {Proposition~\ref{pro_key_g}}, applied with     $g(t)=t^{-1+\vge}$ implies 
	\begin{align*}
		\sum_X m &W_{1+\varepsilon} G_f^{\gamma p} \\ 
		\leq & \gamma^{p-1} \sum_X m f G_f^{\varepsilon} + C \sum_{x,y\in X} b(x,y)  (\nabla_{xy} G_f)^{\langle p-1 \rangle}
		{\int_{G_f(y)}^{G_f(x)}  t^{-1+\varepsilon}{\dt}}\\  
		\leq &  {\gamma^{p-1} \sum_X m f G_f^{\varepsilon} +  C \int_0^1 t^{-1+\varepsilon}{\dt}<\infty},
	\end{align*}
	where the boundedness of the first term follows from the assumption $f\in\ell^1(X,m)$ and boundedness of $G_f^\varepsilon$ due to its convergence to zero.
	
	{\em Null-criticality for $W_{1}$:} Recall that $v:=G_f^{({p-1})/{p}}$ is the Agmon ground state of the operator $Q'_{b,V-W_{1}}$.  We again apply the formula for $W_{1}$ from Lemma~\ref{lem_rep_W1ep} and use the fact that $F_0 $ is bounded from below if $G_f$ satisfies the bounded oscillation condition by Lemma~\ref{lem_Fbounded}. This together with the {lower} coarea estimate, {Proposition~\ref{pro_key_g}}, applied with     $g(t)=t^{-1 }{1_{(0,M_f)}(t)}$  implies 
	\begin{align*}
		\sum_X m W_{1 } G_f^{\gamma p} 
		%=& \gamma^{p-1} \sum_X m f G_f^{\varepsilon} + \sum_{x\in X}  \sum_{y \sim x} b(x,y) (\nabla_{xy} G_f)^{\langle p-1 \rangle} {F}_\varepsilon(G_f(x), G_f(y))\\
		\geq & \gamma^{p-1} \sum_X m f   + C \sum_{x,y\in X}   b(x,y)  (\nabla_{xy} G_f)^{\langle p-1 \rangle}
		{\int_{G_f(y)}^{G_f(x)}  t^{-1 }{\dt}}\\  
		\geq &  {\gamma^{p-1} \sum_X m f   +  C \int_0^{{M_f}} t^{-1 }{\dt}=\infty},
	\end{align*}
	which shows null-criticality and, therefore, finishes the proof.
\end{proof}
It still remains to prove the boundedness of the functions $a_k$ and $\tilde{a}_k$.

\begin{lemma}\label{lem:a_bounded}
	Let $1<p<\infty$, and $\varepsilon \in [0,1]$. Then the functions $a_k$ and $\tilde{a}_k$ defined in the proof of Theorem~\ref{t:graphs} are bounded from above by a constant independent of $k$.
\end{lemma}
\begin{proof} We  let $\gamma  = \frac{p-1+\varepsilon}{p}$  as in the proof of Theorem~\ref{t:graphs}. We set $a=G_f(x)$ and $b=G_f(y)$ for brevity. We assume without loss of generality that $a\geq b$ and also $a^\gamma\geq 1/k^2$ and $b^\gamma\le 1/k$ since $a_k(x,y)=0$ otherwise. We also  set
	\begin{align*}
		A=\min\{1/k,a^{\gamma}\},\qquad B=\max\{1/k^2,b^{\gamma}\}.
	\end{align*}
	With this notation, we have $w_k(x)= 2+\log A/\log k$ and $w_k(y)= 2+\log B/\log k$, where $w_k$ is given by \eqref{wk}. In particular, we have 
	%with $C_{p,\varepsilon} = (\gamma_\varepsilon/\gamma_0)^{p}$ that
	\begin{align*}
		|\nabla_{xy} w_k|^p= \frac{\log^p A/B}{\log^p k}.
	\end{align*}
	With the change of variable of $\tau=t^{\gamma}$  and the observation that $t\varphi_k'(t)=1_{[1/k^2,1/k]} /\log k$, we have
	\begin{align*}
		\int_{G_f(y)}^{G_f(x)} t^{-1+\gamma p }|\varphi_k'(t^{\gamma})|^p \dt&=\frac{1}{{\gamma}} \int_{b	^\gamma}^{a^{\gamma}} \tau^{p-1}|\varphi_k'(\tau)|^p \mathrm{d}\tau= \frac{1}{{\gamma}\log^p k}\log \frac AB \,.
	\end{align*}
	Since $\log A/B\leq \log a^\gamma/b^\gamma\leq  \log a/b$ as $ \gamma\le 1 $, and $ab\leq 1$, we have with $t=a/b$ 
	\begin{multline*}
		a_k(x,y)  = \frac{\left({v}(x){v}(y)\right)^{\frac{p}{2}} |\nabla_{xy} w_k|^{p}}
		{(\nabla_{x,y}G_f)^{\langle p-1\rangle}\left( \int_{G_f(y)}^{G_f(x)} t^{-1  +\gamma p }|\varphi_k'(t^{\gamma}	)|^p \dt\right)} \\[2mm]
		= \frac{{\gamma}(ab)^{\frac{p-1+\varepsilon}{2}} \log^{p-1} \frac AB}{(a-b)^{ p-1} } \leq
		\frac{(ab)^{\frac{p-1}{2}} \log^{p-1}\frac ab}{(a-b)^{ p-1} }
		=t^{\frac{p-1}{2}} \left(\frac{\log t}{t-1}\right)^{ p-1} \!\!=:f(t).
	\end{multline*}
	The function $f$ is continuous on $(1,\infty)\setminus\{1\}$ and has finite limits as $t\to 1$ and $\infty$. Hence, it is bounded from above by a constant independent of $k$. 
	
	The proof for $\tilde{a}_k$ is similar, and we use the notation above. We assume without loss of generality that $a\geq b$ and   $a^\gamma\geq 1/k^2$  since $\tilde a_k(x,y)=0$ otherwise. 
	We compute
	\begin{align*}
		\int_{b}^{a} t^{-1  }1_{ [1/k^{2 },1]}(t^\gamma) \dt=\frac{1}\gamma\int_{b^\gamma}^{a^\gamma} \tau^{-1  }1_{ [1/k^{2 },1]}(\tau)d\tau =\frac{1}{\gamma}\log a^\gamma/B.
	\end{align*}
	If $B=b^\gamma$, then estimating $w_k\leq 1$ and letting $t=a/b$, we get
	\begin{multline*}
		\tilde a_k(x,y) =\frac{\left|\nabla_{x,y}v \right|^p \left(\frac{ |w_k(x)|+|w_k(y)|}{2}\right)^p}
		{(\nabla_{x,y}G_f)^{\langle p-1\rangle} \left( \int_{G_f(y)}^{G_f(x)} t^{-1  }1_{ [1/k^{2 },1]}(t^\gamma) \dt\right) }\\[2mm]
		\leq C  \frac{(a^\gamma - b^\gamma)^p}{(a - b)^{p-1} \log(a/b)}
		=C b^{\gamma p - p + 1}  \frac{(t^\gamma - 1)^p}{(t - 1)^{p-1} \log t}\\[2mm]
		\leq C t^{-(\gamma p - p + 1)}  \frac{(t^\gamma - 1)^p}{(t - 1)^{p-1} \log t}=:g(t),
	\end{multline*}
	since $b=a/t\le 1/t$ and $ \gamma p-p+1\geq 0 $. 
	The function on the right-hand side is continuous in $(1,\infty)$ and has finite limits as $t\to 1$ and $\infty$. Hence, it is bounded from above by a constant independent of $k$.
	For $B=1/k^2$, we have
	\begin{align*}
		\frac{  \left(\frac{ |w_k(x)|+|w_k(y)|}{2}\right)^p}
		{  \int_{b}^{a} t^{-1  }1_{ [1/k^{2 },1]}(t^\gamma) \dt}\leq \gamma \frac{  |w_k(x)|  }{\log (k^2a^\gamma)} 
		{\leq} \gamma\frac{\log  (k^2 a^\gamma)}{(\log  k)\log (k^2a^\gamma)} 
		=\frac{\gamma}{\log k}\leq C \,.
	\end{align*}
	So
	
	for $t=a/b$ we get 
	\begin{align*}
		\tilde a_k(x,y)  
		\leq C  \frac{(a^\gamma - b^\gamma)^p}{(a - b)^{p-1}  }
		\leq C t^{-(\gamma p - p + 1)}  \frac{(t^\gamma - 1)^p}{(t - 1)^{p-1} }=:h(t),
	\end{align*}
	which is again continuous and bounded on $(1,\infty)$. 
\end{proof}
%\Hmm{M: I included the AI proof send by Yehuda. It was polished in coperation with Chat-GPT6-Astra.}

{\subsection{The case $V=0$}}

		If $V=0$, we can extend our considerations to the case
$1-p<\vge<0$, provided that the Green potential satisfies
the bounded oscillation assumption.

\begin{theorem}\label{thm_v0d}
Let $1<p<\infty$. Suppose that $Q_{b,0}$ is subcritical in $X$,
and let $G_f$ be a Green potential of $Q'_{b,0}$ with charge
$0\lneqq f\in\ell^1(X,m)$. Assume that
\[
\lim_{x\to\bar{\infty}}G_f(x)=0.
\]
For $1-p<\vge\leq1$, let
\[
\gamma=\gamma_\vge:=\frac{p-1+\vge}{p}\,,
\qquad
W_{1+\vge}:=
\frac{Q'_{b,0}[G_f^\gamma]}{G_f^{\gamma(p-1)}}\,.
\]
Then $W_{1+\vge}$ is a Hardy weight for $Q_{b,0}$,
which is strictly positive if $f>0$. Moreover, the functional
$Q_{b,-W_{1+\vge}}$ is
\begin{enumerate}
\item[(a)]
subcritical in $X$ for all $1-p<\vge<0$, provided that
\[
\sup_{x\sim y}\frac{G_f(x)}{G_f(y)}<\infty;
\]
\item[(b)]
critical in $X$ with ground state $G_f^\gamma$
for all $0\leq\vge\leq1$;
\item[(c)]
positive-critical with respect to $W_{1+\vge}$
for all $0<\vge\leq1$;
\item[(d)]
null-critical with respect to $W_1$ for $\vge=0$,
provided that $G_f$ satisfies the bounded oscillation assumption.
In particular, under this assumption, if $f>0$, then $W_1$
is a strictly positive optimal Hardy weight for $Q_{b,0}$.
\end{enumerate}
\end{theorem}

\begin{lemma}\label{lem:charge_estimate}
Let $1<p<\infty$, and let $G_f$ be a Green potential of
$Q'_{b,0}$ with charge $0\lneqq f\in\ell^1(X,m)$.
Assume that
\[
\lim_{x\to\bar{\infty}}G_f(x)=0,
\qquad
\sup_{x\sim y}\frac{G_f(x)}{G_f(y)}<\infty.
\]
For $0<\gamma<(p-1)/p$, set $v:=G_f^\gamma$.
Then there exists $C>0$ such that
\[
\sum_X mfw
\leq C\left(
Q^v_{\mathrm{sim}}(w)^{1/2}
+
Q^v_{\mathrm{sim}}(w)^{1/p}
\right)
\]
for all $w\in C_c(X)$ with $0\leq w\leq1$.
\end{lemma}
\begin{proof}
	We start with two observations.

(a)	By bounded oscillation and Proposition~\ref{pro_key_g},
for  $q>-1$ we have
\begin{multline*}
T_q:=\sum_{G_f(x)>G_f(y)}
b(x,y)|\nabla_{xy}G_f|^p G_f(y)^q\\
\leq
C\sum_{G_f(x)>G_f(y)}
b(x,y)|\nabla_{xy}G_f|^{p-1}
\int_{G_f(y)}^{G_f(x)}t^q\,\dt
\leq C\int_0^{\sup_XG_f}t^q\,\dt<\infty.
%\label{eq:charge_summability}
\end{multline*}
%Here and below, the sums are over $x,y\in X$ satisfying the indicated inequality, so each edge on which $G_f$ is nonconstant is counted once.

(b) For $w\in C_c(X)$ with $0\leq w\leq1$, we obtain from
testing  $Q'_{b,0}[G_f]=f$ with $w$ 
\begin{align*}
0\leq\sum_Xmfw
%&=\sum_{G_f(x)>G_f(y)}b(x,y)|\nabla_{xy}G_f|^{p-1}\nabla_{xy}w
%\nonumber\\
&\leq
\sum_{G_f(x)>G_f(y)}
b(x,y)|\nabla_{xy}G_f|^{p-1}|\nabla_{xy}w|.
\label{eq:charge_green}
\end{align*}
%All sums in this identity converge absolutely since $w$ has finite support and $G_f$ is bounded.

We now distinguish two cases.

Suppose first $p\geq2$. By bounded oscillation
and the definition of the simplified Lagrangian,
\[
S^v(w)_{xy}
\geq C^{-1}b(x,y)G_f(y)^{\gamma p}
|\nabla_{xy}w|^p
\]
whenever $x\sim y$ and $G_f(x)>G_f(y)$.
Hence, H\"older's inequality gives
\begin{align*}
\sum_Xmfw
&\leq C Q^v_{\mathrm{sim}}(w)^{1/p} T_{-\gamma p/(p-1)}^{(p-1)/p}
%\left( \sum_{G_f(x)>G_f(y)} b(x,y)|\nabla_{xy}G_f|^p G_f(y)^{-\frac{\gamma p}{p-1}} \right)^{\frac{p-1}{p}}.
\end{align*}
The sum is finite by observation (a) %\eqref{eq:charge_summability},
since $-\gamma p/(p-1)>-1$.

Suppose now  $1<p<2$.
Bounded oscillation and the mean value theorem imply
$v(x)\asymp v(y)$ and
$|\nabla_{xy}v|\asymp
G_f(y)^{\gamma-1}|\nabla_{xy}G_f|$
on the edges under consideration.
Using also $w(x)+w(y)\leq2$, we obtain
\[
S^v(w)_{xy}\geq
\frac{1}{C} b(x,y)G_f(y)^{\gamma p}
\begin{cases}
\left(\dfrac{|\nabla_{xy}G_f|}{G_f(y)}\right)^{p-2}
\!\!|\nabla_{xy}w|^2,
&
\!|\nabla_{xy}w|
\leq \dfrac{|\nabla_{xy}G_f|}{G_f(y)},\\[3mm]
|\nabla_{xy}w|^p,
&
\!|\nabla_{xy}w|
> \dfrac{|\nabla_{xy}G_f|}{G_f(y)}\,.
\end{cases}
\]
Split the last sum in observation (b)
according to these two cases.
Applying the Cauchy--Schwarz inequality to the first
part and H\"older's inequality to the second gives
\begin{align*}
\sum_Xmfw
\leq C \left(Q^v_{\mathrm{sim}}(w)^{1/2} T_{p-2-\gamma p}^{1/2}+Q^v_{\mathrm{sim}}(w)^{1/p} T_{-\gamma p/(p-1)}^{1/p}\right).
%\left(\sum_{G_f(x)>G_f(y)} b(x,y)|\nabla_{xy}G_f|^p G_f(y)^{p-2-\gamma p}\right)^{1/2}\\
%+ C Q^v_{\mathrm{sim}}(w)^{1/p}\left(\sum_{G_f(x)>G_f(y)}b(x,y)|\nabla_{xy}G_f|^pG_f(y)^{-\frac{\gamma p}{p-1}}\right)^{\frac{p-1}{p}}.
\end{align*}
Both sums are finite by the estimate  in the beginning %\eqref{eq:charge_summability},
since
\(
p-2-\gamma p>-1,
\) and \(
-{\gamma p}/({p-1})>-1.
\)
Absorbing them into the constant proves the assertion.
\end{proof}

\begin{proof}[Proof of Theorem~\ref{thm_v0d}]
For $1-p<\vge\leq1$, let $v:=G_f^\gamma$.
Since $G_f$ vanishes at infinity, it is bounded,
and hence $v\in\mathcal F(X)$.
By concavity of $t\mapsto t^\gamma$ and the equation
$Q'_{b,0}[G_f]=f$, we have
\[
Q'_{b,0}[G_f^\gamma]
\geq
\gamma^{p-1}G_f^{(\gamma-1)(p-1)}f.
\]
Consequently,
\[
W_{1+\vge}
\geq \gamma^{p-1}\frac{f}{G_f^{p-1}}\geq0.
\]
Thus $W_{1+\vge}$ is nontrivial and is strictly positive
if $f>0$. Since $Q'_{b,-W_{1+\vge}}[v]=0$,
the AAP theorem \cite[Theorem~2.3]{Fischer1} implies
$Q_{b,-W_{1+\vge}}\geq0$.
Assertions (b)--(d) follow from Theorem~\ref{t:graphs}.

It remains to prove (a).
Fix $1-p<\vge<0$, and assume that $G_f$ satisfies
the bounded oscillation assumption.
%By Lemma~\ref{lem:charge_estimate} and Proposition~\ref{Prop:simp_energyg}, \begin{equation}\label{eq:subcrit_charge_estimate} \sum_Xmfw \leq C\left( Q_{b,-W_{1+\vge}}(vw)^{1/2} + Q_{b,-W_{1+\vge}}(vw)^{1/p} \right) \end{equation} for every $w\in C_c(X)$ with $0\leq w\leq1$.

Suppose that $Q_{b,-W_{1+\vge}}$ is critical.
Choose $o\in X$ with $f(o)>0$, and let $(K_k)$
be an increasing exhaustion of $X$ by finite connected
sets containing $o$.
By \cite[Theorem~2.6 and the proof of
Proposition~5.5]{Fischer1}, there exist nonnegative
capacitary minimizers $\psi_k\in C_c(K_k)$ such that
\[
\psi_k(o)=v(o),\qquad
Q'_{b,-W_{1+\vge}}[\psi_k]=0
\quad\mbox{in }K_k\setminus\{o\},
\qquad
Q_{b,-W_{1+\vge}}(\psi_k)\to 0.
\]
The weak comparison principle
\cite[Proposition~5.3(b)]{Fischer1}, applied on
$K_k\setminus\{o\}$, gives $0\leq\psi_k\leq v$.
Indeed, $\psi_k=v$ at $o$, $\psi_k=0\leq v$
outside $K_k$, and both functions solve the same
equation on $K_k\setminus\{o\}$.

Set $w_k:=\psi_k/v$. Then $w_k\in C_c(X)$,
$0\leq w_k\leq1$, and $w_k(o)=1$.
Applying Lemma~\ref{lem:charge_estimate}, we obtain
\[
0<m(o)f(o)
\leq \sum_Xmfw_k
\leq C\left(
Q_{b,-W_{1+\vge}}(\psi_k)^{1/2}
+
Q_{b,-W_{1+\vge}}(\psi_k)^{1/p}
\right)\to 0,
\]
which is a contradiction. Hence $Q_{b,-W_{1+\vge}}$
is subcritical in $X$.
\end{proof}

% \section{Rellich inequality and an open problem in the linear case}\label{sec_lin}
\section{The linear case: applications \& open questions}\label{sec_lin}

{In the present section we consider the linear case, where $p=2$, in the continuum setting. First we extend the results of Section~\ref{sec_constr} without the extra assumption that $Q_{A,c_\varepsilon V}$ is subcritical. Then we discuss Rellich-type inequalities and conclude with open questions.}

\subsection{Construction of Hardy weights for subcritical $Q_{A,V}$}
As we have seen in the purely quasilinear case (i.e., $p \neq 2$), in order to construct strictly positive optimal Hardy weights for $Q_{A,V}$, we require the subcriticality of $Q_{A,c_\varepsilon V}$, with $c_\varepsilon>1$, rather than that of $Q_{A,V}$.
In the linear case (i.e., $p=2$), we can overcome this gap using the ground state transform.

{Assume that $p=2$.} Let the energy functional $Q=Q_{A,V}$, corresponding to the Schr\"odinger operator $Q'=Q_{A,V}'$, be subcritical in $\Gw$ with minimal positive  Green function $G$. By the AAP theorem, there exists a positive {function $h\in W^{1,2}_\loc(\Gw)$ satisfying $Q_{A,V}'[h]=0$ in $\Gw$.
The simplified energy $Q_{\mathrm{sim}}^h$ with respect to $h$ and $\dy$ is given by
\[\tilde{Q}(w):= Q_{\mathrm{sim}}^h(w)= Q(hw)=\int_{\Gw} h^2 |\nabla w|_A^2 \dy . \]
The subcritical symmetric operator (with respect to $\dy$) associated with the energy form $\tilde{Q}$ is given by       
    \[ \tilde{Q}'(w):= -\mathrm{div}(h^2 A \nabla w)=hQ_{A,V}'(hw) \,,\]
  and the corresponding minimal positive Green function, with respect to $\dy$, is given by $$\widetilde{G}(x,y):=\frac{G(x,y)}{h(x)h(y)} \qquad  x,y\in \Gw.$$
  Therefore, the Green potential of $\tilde{Q}'$ with a charge $f$ is given by
  $$\tilde{G}_f(x):=
\frac1{h(x)}\int_\Gw G(x,y)\frac{f(y)}{h(y)}\dy.$$
  Let $f\geq 0$ be such that $G_{f/h}$ is a Green potential for $Q'$ with a charge $f/h$. Assume that the Green potential of $\tilde{Q}'$ with charge $f$ satisfies   
$$\lim\limits _{x\to \bar{\infty}}\tilde{G}_f(x)= \lim\limits _{x\to \bar{\infty}}\frac1{h(x)}\int_\Gw G(x,y)\frac{f(y)}{h(y)}\dy=0.$$} 
Consider the weight functions, for $\varepsilon\in (-1,1]$,
{\[\tilde\omega_{1+\vge}:= \frac{\tilde{Q}'[\tilde{G}_f^{(1+\vge)/2}]}{\tilde{G}_f^{(1+\vge)/2}}= \frac{1+\vge}{2}\left(\frac{1-\vge}{2}  \ \frac{h^2 |\nabla \tilde{G}_f|_A^2}{\tilde{G}_f^2} + \frac{f}{\tilde{G}_f}\right)   \,.
\]} 
By applying Theorem \ref{thm_v0} to $\tilde{Q}$, 
we infer that  
$$\tilde{Q}_{-\tilde{\omega}_{1+\vge}}(w):=\tilde{Q}(w)-\int_{\Gw} \tilde{\omega}_{1+\vge}|w|^2\dx$$ 
satisfies the conclusions $(a)$--$(d)$ of Theorem \ref{thm_v0}. Hence, the same holds for  $Q_{A,V-\tilde{\omega}_{1+\vge}/h^2}$ as $\tilde{Q}(w)=Q(hw)$ for all $w \in W^{1,2}(\Gw) \cap C_c(\Gw)$.

Therefore, we have the following result.
\begin{theorem}
    Let  $p=2$, {and suppose that $Q_{A,V}$} is subcritical in $\Gw$ with a positive solution $h\in W^{1,2}_\loc(\Gw)$ of the equation $Q_{A,V}'[h]=0$ in $\Gw$. Let $\tilde{G}_f$ be the Green potential of $\tilde{Q}'$ with charge  $0\lneqq f\in L^1(\Gw) \cap L^{{q}}_\loc(\Gw)$, {$q > n$}. Then, for $\varepsilon\in (-1,1]$,
  {\[\omega_{1+\vge}\!:= \frac{\tilde{\omega}_{1+\vge}}{h^2}=     \frac{1+\vge}{2}\left(\frac{1-\vge}{2}  \ \frac{ |\nabla \tilde{G}_f|_A^2}{\tilde{G}_f^2} + \frac{f}{h^2\tilde{G}_f}\right)\] }
    is a Hardy weight for $Q_{A,V}$ in $\Gw$ which is strictly positive if $f>0$. Assume further $$\lim\limits _{x\to \bar{\infty}}\tilde{G}_{f}=0.$$
Then, the functional $Q_{A,V-\omega_{1+\vge}}$ is
\begin{enumerate}
	\item[(a)]   subcritical in $\Gw$ for all $-1< \vge <0$,
	\item[(b)]    critical in $\Gw$ with ground state {$\big(G_{f/h}\big)^{(1+\vge)/2}$}  for all $0\leq \vge \leq 1$,
	\item [(c)]    positive-critical in $\Gw$ with respect to $\omega_{1+\vge}$  for all $0< \vge \leq 1$,
	\item[(d)]  null-critical in $\Gw$ with respect to $\omega_{1}$ for $\vge=0$. {In particular, if $f>0$, then $\omega_{1}$ is a strictly positive optimal Hardy weight for $Q_{A,V}$ in $\Gw$.}
\end{enumerate} 
\end{theorem}

\subsection{Rellich-type inequalities}
 {In \cite{DFP}, the authors obtain results on the spectrum of the symmetric operator $W^{-1}P$, acting on the weighted space $L^2(\Gw\setminus \{x_0\},W\!\dx)$, where $P=Q_{A,V}$ is a Schr\"odinger-type operator, and $W$ is a strictly positive optimal Hardy weight derived by the supersolution construction of the Green function} rather than the Green potential as in this work. In particular, they show that the spectrum and the essential spectrum of $W^{-1}P$ are both $[1,\infty)$.

To explain the problem in the simplest case, we consider the {\em linear} symmetric elliptic operator 
\begin{equation*}%\label{P}
	Pu:=-\div\big(A \nabla u\big),
	\end{equation*}
where the values of $A=\big[a^{ij}\big]\in {C(\Gw,\R^{n\times n})}$ are symmetric locally uniform positive definite matrices.  We assume that $P$ is subcritical in $\Gw$, which implies that $P$ admits a minimal positive Green function $G $  in $\Gw$. Let  $0\lneqq  f\in L^{2}_\loc(\Gw){\cap L^1(\Gw)}$, and assume that for some (all) $x\in\Gw$
$${G_f (x) :=\int_\Gw G(x,y)f(y)\dy <\infty \quad\mbox{ and }\quad \lim\limits _{x\to \bar{\infty}}{G}_{f}(x)=0.}$$
 Using an exhaustion argument,  it follows that $G_f$ is the minimal positive solution of the equation $Pu=f$ in $\Gw$. 
Hence, $G_f$   is the Green potential of $P$ in $\Gw$ with charge $f$.

As in \cite{DFP,Grillo}, the strictly positive optimal Hardy weight can be used to obtain Rellich-type inequality.  To this end, one needs to find a solution to the following {\em strict eikonal  inequality }
\begin{equation*}%\label{h_cond}
{|\nabla \gk|^2_A}< W \qquad {\mbox{ a.e. in } \Gw,}
\end{equation*}
where $W$ is a Hardy weight (see \cite{Ag82}).  Such a {\em strictly positive} optimal Hardy weight provides us with a natural candidate for $
W$ and $\gk$. Assume that our optimal Hardy weight $W_1$ of $P$ is obtained using Theorem~\ref{pro_opt} with a Green potential $G_f$ of $P$ with $f>0$ and a ground state $v=G_f^{1/2}$.  Set
$$W_1= \frac{1}{4} \frac{|\nabla G_f |_A^2}{G_f ^2}+\frac{1}{2}\!\frac{f}{G_f}, \qquad 
\gk :=\frac{1}{2}\log G_f.$$ 
Then $W_1$ and $\gk$  clearly satisfy the above strict eikonal inequality.

Consider the Riemannian metric 
\begin{equation*}\label{eq_metric}
\qquad {\ds^2}:= W_1(x)\sum_{i,j=1}^{n} a_{ij}(x)\dx^i\dx^j, \quad \mbox{ where } \big[a_{ij}\big]:= \big[a^{ij}\big]^{-1}=A^{-1},
\end{equation*}
{and $W_1\in C(\Gw)$.} Denote by $\rho$ the induced geodesic distance to some fixed point $x_0\in \Gw$. Since 
 $$ {|\nabla \gk|^2_A} =   \frac{1}{4}\frac{|\nabla G_f |_A^2}{G_f ^2} < W_1 \qquad \mbox{ a.e. in } \Gw, $$ it follows from the proof of \cite[Lemma~10.1]{DFP} that the above Riemannian metric is {\em complete}. Exploiting \cite[Theorem~1.5]{Ag82}, we obtain the following \textit{Rellich-type inequality}. 
 \begin{corollary}[Rellich inequality]\label{cor_rellich}
Suppose that {$u\in W^{1,p}_\loc(\Gw)$} satisfies $Pu \in L^2_{\loc}(\Gw)$ and the growth condition
$$\int_{\Gw} |u|^2 W_1\exp\left(-2(1-\gd)\gr\right) \dx<\infty$$
for some $\gd>0$, then
\begin{equation*}\label{exp3}
	\int_{\Gw} |u|^2f \dx\leq
	4\int_{\Gw} \frac{|Pu|^2}{f}(G_f)^{2} \dx\,,
\end{equation*}
where $G_f$ is a Green potential  of $P$ with $f>0$. 
\end{corollary}
\begin{proof}
	The result is a special case of  Agmon's estimate (\cite[Theorem~1.5]{Ag82})
	\begin{equation*}
\int_{\Gw} |u|^2(W- {|\nabla\kappa|_A^2})e^{2\kappa} \dx\leq
\int_{\Gw}|Pu|^2 (W-{|\nabla\kappa|_A^2})^{-1}e^{2\kappa}  \dx\,
\end{equation*}
	with $W=W_1=\frac{1}{4} \frac{|\nabla G_f |_A^2}{G_f ^2}+\frac{1}{2}\!\frac{f}{G_f}$ and $\gk=\frac{1}{2}\log G_f$ which gives $( W- {|\nabla \kappa|_A^2})e^{2\kappa}=f/2$ and $( W-|\nabla \kappa|_A^2)^{-1}e^{2\kappa}= 2G_f^2/f$.
\end{proof}
We note that using the ground state transform, one can easily extend the above Rellich inequality to a corresponding statement in the case of a general subcritical Schr\"odinger operator $P$ with potential $V$.

\subsection{Open questions}
We conclude the paper with open questions. 
For the minimal positive Green function $G(x)=G(x,x_0)$ of the subcritical operator $P$ in $\Gw$ with singularity at $x_0$, we denote by 
$$W_0:=\frac{1}{4} \frac{|\nabla G |_{A}^2}{G ^2}$$
on $\Gw^*:= \Gw\setminus \{x_0\}$.  
In \cite{DFP} the authors proved that if $W_0>0$, then the spectrum and the essential spectrum of the Friedrichs extension of $W_0^{-1}P$ on $L^2(\Gw^*, W_0\dx)$ are  equal to  $[1,\infty)$. Since we have $W_1>0$ for $f>0$, it is natural to ask whether the same holds for $W_1= \frac{|\nabla G_f |_A^2}{4G_f ^2} + \frac{f}{2G_f}$ {in $\Gw$}.
%\medskip

 \begin{problems} We pose the following questions.\begin{enumerate}
	\item  Are the spectrum and the essential spectrum of the Friedrichs extension of  $W_1^{-1}P$ on $L^2(\Gw, W_1\dx)$ both equal to  $[1,\infty)$? 
	\item Does $(1,\infty)$ consist of purely {absolutely} continuous spectrum?
 \end{enumerate}
\end{problems}
{\bf Discussion:} We discuss where the strategy of \cite{DFP} fails in our setting. Note that using the supersolution construction $G_f^\ga$ for spectral parameter $$\gl=4\ga(1-\ga)>1\quad\mbox{ with }\quad\ga= \frac{1+ \sqrt{1-\gl}}{2}$$ and $f>0$, one obtains a complex valued ``Hardy weight" 
$$W_\gl= \frac{|\nabla G_f |_A^2}{4G_f ^2} + \frac{\ga f}{\gl G_f} .$$   
%Indeed, let   $f>0$ such that  $G_f$ is its Green potential. 
For $\gl {\geq 1}$, the function $ u_\gl:=G_f^\ga$ is a {nontrivial solution of} $(P-\gl W_\gl)u_\gl=0$. 
We note that the behavior of $u_\gl$ and $W_\gl$ depends on the value of $\gl$. In particular, we have the following cases:
%\medskip
\begin{itemize}
	\item For $ \gl\leq 1$, we have  $ u_\gl>0$.
	\item For $\gl>1$,   
	$\ga= \frac{1\,+\, \mathrm{i}\sqrt{\gl-1}}{2}$. 
	So, $u_\gl=u_{1+4\xi^2}= G_f^{{1}/{2}+\mathrm{i}\xi} $,  satisfies
	$$ 
	 \left({P}-(1+4\xi^2)W_\gl \right)
		u_{\gl}=0,  $$
		where   {$\xi=\frac{\sqrt{\gl-1}}{2}\geq 0$,} and $W_\gl$, $u_\gl$ are complex valued functions.
\end{itemize}
Note that  for $\gl\geq 1$ and the Agmon ground state $u_1$, we have
$$|u_\gl|= u_1=G_f^{{1}/{2}}.$$ 
In the case, when  we consider the minimal positive Green function  $G$  of $P$ instead of $G_f$, the Hardy weight does not depend on $\gl$, so  $W_0=W_\gl$ from above. Therefore, if $W_0>0$, the Shnol-type theorem of Beckus and Pinchover \cite{BP} (see also \cite{DFP}) implies that  these $\gl\geq 1$ are in the spectrum of $W_0^{-1}P$, and the spectrum and essential spectrum are equal to  $[1,\infty).$

However, for $G_f$ with $f>0$, the potential {$W_\gl\neq W_0$} is complex valued for $\gl>1$, and therefore, the corresponding operator is not self-adjoint. Consequently, the above Shnol-type theorem cannot be applied, and the question about the spectrum of $W_1^{-1}P$ is open.\medskip

%%%%%%%%%%
%%%%%%%%%%%%%%%%%%%%%%%%%%%%%%%%%%%%%%%%%
%\begin{center}
	{\bf Acknowledgments.} % \end{center}
The authors thank Shubham Gupta for valuable discussions.  The authors acknowledge several inspiring discussions with AI models concerning technical estimates in the discrete case (Section~\ref{sec_disc}). They also credit these models with contributions to the proofs of Theorem~\ref{thm_v0d} and Lemma~\ref{lem:charge_estimate} and with assistance in proofreading the manuscript.  
This research was supported by ISF-DFG Lead Agency research grant (ISF's No. 1531/25). MK and YP thank the Israel Institute of Advanced Studies in Jerusalem for its hospitality.
%%%%%%%%%%%%%%%%%%%%%%%%%%%%%%%%%%%%%%%%%%%%%%%%%%%%%%%%%%%%

\end{document}